\documentclass{amsart}
\usepackage{amssymb}
\usepackage[mathcal]{euscript}
\usepackage[cmtip,all]{xy}

\newcommand{\bydef}{:=}

\newcommand{\vphi}{\varphi}
\newcommand{\veps}{\varepsilon}

\newcommand{\wh}[1]{\widehat{#1}}
\newcommand{\wt}[1]{\widetilde{#1}}
\newcommand{\wb}[1]{\overline{#1}}

\newcommand{\sym}{\mathcal{H}}
\newcommand{\sks}{\mathcal{K}}
\newcommand{\id}{\mathrm{id}}

\newcommand{\lspan}[1]{\mathrm{span}\left\{#1\right\}}
\newcommand{\matr}[1]{\left(\begin{smallmatrix}#1\end{smallmatrix}\right)}
\newcommand{\diag}{\mathrm{diag}}

\DeclareMathOperator*{\ot}{\otimes}
\newcommand{\tr}{\mathrm{tr}}

\DeclareMathOperator{\rank}{\mathrm{rank}} 

\newcommand{\bi}{\mathbf{i}}

\newcommand{\cA}{\mathcal{A}}
\newcommand{\cB}{\mathcal{B}}
\newcommand{\cC}{\mathcal{C}}
\newcommand{\cD}{\mathcal{D}}

\newcommand{\cH}{\mathcal{H}}
\newcommand{\cI}{\mathcal{I}}
\newcommand{\cJ}{\mathcal{J}}
\newcommand{\cK}{\mathcal{K}}
\newcommand{\cL}{\mathcal{L}}

\newcommand{\cO}{\mathcal{O}}
\newcommand{\cQ}{\mathcal{Q}}

\newcommand{\cU}{\mathcal{U}}
\newcommand{\cV}{\mathcal{V}}

\newcommand{\frs}{{\mathfrak s}}

\newcommand{\frn}{{\mathfrak n}}

\DeclareMathOperator{\CD}{\mathfrak{CD}}

\newcommand{\ZZ}{\mathbb{Z}}

\newcommand{\QQ}{\mathbb{Q}}

\newcommand{\FF}{\mathbb{F}}

\newcommand{\chr}[1]{\mathrm{char}\,#1}

\DeclareMathOperator{\End}{\mathrm{End}}

\DeclareMathOperator{\im}{\mathrm{im}\,}
\DeclareMathOperator{\Aut}{\mathrm{Aut}}
\DeclareMathOperator{\Der}{\mathrm{Der}}
\DeclareMathOperator{\inder}{\mathrm{IDer}}
\DeclareMathOperator{\supp}{\mathrm{Supp}\,}

\newcommand{\Jord}[1]{{#1}^{(+)}}

\newcommand{\Ad}{\mathrm{Ad}}

\newcommand{\Gl}{\mathfrak{gl}}

\newcommand{\frso}{{\mathfrak{so}}}

\newcommand{\GL}{\mathrm{GL}}
\newcommand{\SL}{\mathrm{SL}}

\newcommand{\subo}{_{\bar 0}}
\newcommand{\subuno}{_{\bar 1}}
\newcommand{\alb}{\mathbb{A}} 
\newcommand{\HQ}{{\sym}_4(\cQ)}
\newcommand{\HK}{{\sym}_4(\cK)}
\newcommand{\zero}{{\bar{0}}}
\newcommand{\one}{{\bar{1}}}
\newcommand{\two}{{\bar{2}}}
\newcommand{\three}{{\bar{3}}}
\newcommand{\pf}{\mathrm{pf}}
\newcommand{\invol}{\,\bar{ }\,} 
\newcommand{\str}{\mathfrak{str}}
\newcommand{\stu}{\mathfrak{stu}}
\newcommand{\kan}{\mathfrak{kan}}

\newcommand{\cX}{\mathcal{X}}

\newtheorem{theorem}{Theorem}
\newtheorem{proposition}[theorem]{Proposition}
\newtheorem{lemma}[theorem]{Lemma}
\newtheorem{corollary}[theorem]{Corollary}

\theoremstyle{definition}
\newtheorem{df}[theorem]{Definition}

\theoremstyle{remark}
\newtheorem{remark}[theorem]{Remark}

\begin{document}

\title[A $\ZZ_4^3$-grading on a $56$-dimensional simple structurable algebra]{A $\ZZ_4^3$-grading on a $56$-dimensional\\ simple structurable algebra\\ and related fine gradings on the\\ simple Lie algebras of type $E$}

\author[D. Aranda]{Diego Aranda${}^\star$}
\address{Departamento de Matem\'{a}ticas
 e Instituto Universitario de Matem\'aticas y Aplicaciones,
 Universidad de Zaragoza, 50009 Zaragoza, Spain}
\email{daranda@unizar.es}

\author[A. Elduque]{Alberto Elduque${}^\star$}
\address{Departamento de Matem\'{a}ticas
 e Instituto Universitario de Matem\'aticas y Aplicaciones,
 Universidad de Zaragoza, 50009 Zaragoza, Spain}
\email{elduque@unizar.es}
\thanks{${}^\star$supported by the Spanish Ministerio de Econom\'{\i}a y Competitividad---Fondo Europeo de Desarrollo Regional (FEDER) MTM2010-18370-C04-02 and of the Diputaci\'on General de Arag\'on---Fondo Social Europeo (Grupo de Investigaci\'on de \'Algebra)}

\author[M. Kochetov]{Mikhail Kochetov${}^\dagger$}
\address{Department of Mathematics and Statistics,
 Memorial University of Newfoundland,
 St. John's, NL, A1C5S7, Canada}
\email{mikhail@mun.ca}
\thanks{${}^\dagger$partially supported by a sabbatical research grant of Memorial University and a grant for visiting scientists by Instituto Universitario de Matem\'aticas y Aplicaciones, University of Zaragoza}



\date{}

\begin{abstract}
We describe two constructions of a certain $\ZZ_4^3$-grading on the so-called Brown algebra (a simple structurable algebra of dimension $56$ and skew-dimension $1$) over an algebraically closed field of characteristic different from $2$ and $3$.
We also show how this grading gives rise to several interesting fine gradings on exceptional simple Lie algebras of types $E_6$, $E_7$ and $E_8$.
\end{abstract}

\maketitle

\section{Introduction}

In the past two decades, there has been much progress in the study of gradings on simple Lie algebras by arbitrary groups --- see the recent monograph \cite{EKmon} and references therein.
In particular, over an algebraically closed field of characteristic $0$, fine gradings have been classified for all finite-dimensional simple Lie algebras except $E_7$ and $E_8$.

The second author has shown in \cite{E13} that, in a sense, such a fine grading splits into two independent gradings: a grading by a free abelian group, which is also a grading by a root system, and a fine grading by a finite group on the corresponding coordinate algebra. For
example, the $\ZZ_2^3$-grading on the algebra of octonions that arises from the three iterations of the Cayley--Dickson doubling process is ``responsible'' not only for a fine $\ZZ_2^3$-grading on $G_2$ but also for fine gradings on $F_4$ by the group $\ZZ\times\ZZ_2^3$ and on $E_r$ by $\ZZ^{r-4}\times\ZZ_2^3$ ($r=6,7,8$). We have a similar picture for the $\ZZ_3^3$-grading on the simple exceptional Jordan algebra (the Albert algebra) that can be obtained from the first Tits construction.

In the classification \cite{DV_e6} of fine gradings on the simple Lie algebra of type $E_6$ over an algebraically closed field of characteristic $0$, among the $14$ fine gradings there is one with universal grading group $\ZZ_4^3$, which has an interesting property: the corresponding quasitorus in the automorphism group of $E_6$ contains an outer automorphism of order $4$ but not of order $2$. A model for this grading in terms of a symplectic triple system is given in \cite[\S 6.4]{EKmon}. On the other hand, it is known that $E_6$ can be realized as the derivation algebra of a certain simple nonassociative algebra with involution $\cA$, where the dimension of $\cA$ is $56$ and the dimension of the space of skew elements of $\cA$ is $1$. This algebra with involution belongs to the class of so-called {\em structurable algebras}, which were introduced by Allison in \cite{A78} as a generalization of Jordan algebras. (Jordan algebras are the structurable algebras whose involution is the identity map.) In fact, the algebra $\cA$ itself goes back to Brown \cite{B63} and for this reason is called the {\em split Brown algebra} in \cite{G01}.

In Draper's ongoing work on fine gradings on the simple Lie algebras of type $E_7$ and $E_8$, there appeared a fine grading on $E_7$ with universal goup $\ZZ_2\times\ZZ_4^3$ and two fine gradings on $E_8$ with universal groups $\ZZ_2^2\times\ZZ_4^3$ and  $\ZZ\times\ZZ_4^3$. According to \cite{E13}, this latter must necessarily be induced by a fine grading with universal group $\ZZ_4^3$ on the Brown algebra. This was the starting point of our investigation.

In this article we construct a $\ZZ_4^3$-grading on $\cA$ in two ways: realizing $\cA$ as the Cayley--Dickson double, in the sense of \cite{AF84}, of the quartic Jordan algebra $\HQ$ (the Hermitian matrices of order $4$ over quaternions) or the structurable matrix algebra, in the sense of \cite{AF84}, of the cubic Jordan algebra $\alb=\sym_3(\cC)$ (the Hermitian matrices of order $3$ over octonions). These constructions actually work over any field containing a fourth root of $1$.

The background on gradings and on structurable algebras (in particular, the split Brown algebra $\cA$) will be recalled in Section \ref{s:preliminaries}, and the two constructions of the $\ZZ_4^3$-grading on $\cA$ will be carried out in Sections \ref{s:H4} and \ref{s:MatrixStructAlgebra}. In Section \ref{s:recognition}, we will establish a ``recognition theorem'' for this grading, which in particular implies that our two models are equivalent. Finally, in Section \ref{s:E}, we will explain how this grading can be used to construct the fine gradings on $E_6$, $E_7$ and $E_8$ mentioned above.

\section{Preliminaries}\label{s:preliminaries}

\subsection{Group gradings on algebras}

Let $\cU$ be an algebra (not necessarily associative) over a field $\FF$ and let $G$ be a group (written multiplicatively).

\begin{df}\label{df:G_graded_alg}
A {\em $G$-grading} on $\cU$ is a vector space decomposition
\[
\Gamma:\;\cU=\bigoplus_{g\in G} \cU_g
\]
such that
$
\cU_g \cU_h\subset \cU_{gh}\quad\mbox{for all}\quad g,h\in G.
$
If such a decomposition is fixed, $\cU$ is referred to as a {\em $G$-graded algebra}.
The nonzero elements $x\in\cU_g$ are said to be {\em homogeneous of degree $g$}, and one writes $\deg_\Gamma x=g$ or just $\deg x=g$ if the grading is clear from the context. The {\em support} of $\Gamma$ is the set $\supp\Gamma\bydef\{g\in G\;|\;\cU_g\neq 0\}$.
\end{df}

If $(\cU,\sigma)$ is an algebra with involution, then we will always assume $\sigma(\cU_g)=\cU_g$ for all $g\in G$.

There is a more general concept of grading: a decomposition $\Gamma:\;\cU=\bigoplus_{s\in S}\cU_s$ into nonzero subspaces indexed by a set $S$ and having the property that, for any $s_1,s_2\in S$ with $\cU_{s_1}\cU_{s_2}\ne 0$, there exists (unique) $s_3\in S$ such that $\cU_{s_1}\cU_{s_2}\subset\cU_{s_3}$. For such a decomposition $\Gamma$, there may or may not exist a group $G$ containing $S$ that makes $\Gamma$ a $G$-grading. If such a group exists, $\Gamma$ is said to be a {\em group grading}. However, $G$ is usually not unique even if we require that it should be generated by $S$. The {\em universal grading group} is generated by $S$ and has the defining relations $s_1s_2=s_3$ for all $s_1,s_2,s_3\in S$ such that $0\ne\cU_{s_1}\cU_{s_2}\subset\cU_{s_3}$ (see \cite[Chapter 1]{EKmon} for details).

It is known that if $\Gamma$ is a group grading on a simple Lie algebra, then $\supp\Gamma$ always generates an abelian subgroup. In other words, the universal grading group is abelian. Here we will deal exclusively with abelian groups, and we will sometimes write them additively. Gradings by abelian groups often arise as eigenspace decompositions with respect to a family of commuting diagonalizable automorphisms. If $\FF$ is algebraically closed and $\chr\FF=0$ then all abelian group gradings on finite-dimensional algebras can be obtained in this way.

Let $\Gamma:\, \cU=\bigoplus_{g\in G} \cU_g$ and $\Gamma':\,\cV=\bigoplus_{h\in H} \cV_h$
be two gradings, with supports $S$ and $T$, respectively.
We say that $\Gamma$ and $\Gamma'$ are {\em equivalent} if there exists an isomorphism of algebras $\vphi\colon\cU\to\cV$ and a bijection $\alpha\colon S\to T$ such that $\varphi(\cU_s)=\cV_{\alpha(s)}$ for all $s\in S$. If $G$ and $H$ are universal grading groups then $\alpha$ extends to an isomorphism $G\to H$.

If $\Gamma:\,\cU=\bigoplus_{g\in G} \cU_g$ and $\Gamma':\,\cU=\bigoplus_{h\in H} \cU'_h$ are two gradings on the same algebra, with supports $S$ and $T$, respectively, then we will say that $\Gamma'$ is a {\em refinement} of $\Gamma$ (or $\Gamma$ is a {\em coarsening} of $\Gamma'$) if for any $t\in T$ there exists (unique) $s\in S$ such that $\cU'_t\subset\cU_s$. If, moreover, $\cU'_t\ne\cU_s$ for at least one $t\in T$, then the refinement is said to be {\em proper}. Finally, $\Gamma$ is said to be {\em fine} if it does not admit any proper refinements.

\subsection{Structurable algebras}

Let $(\cA,\invol)$ be an algebra with involution over a field $\FF$, i.e., $a\mapsto\bar{a}$ is an $\FF$-linear involutive antiautomorphism of $\cA$. We will use the notation
\[
\sym(\cA,\invol)=\{a\in\cA\;|\;\bar{a}=a\}\;\mbox{and}\;\sks(\cA,\invol)=\{a\in\cA\;|\;\bar{a}=-a\}.
\]
If $\chr\FF\ne 2$ then $\cA=\sym(\cA,\invol)\oplus\sks(\cA,\invol)$. The dimension of the subspace $\sks(\cA,\invol)$ will be referred to as the {\em skew-dimension} of $(\cA,\invol)$.

\begin{df}
Suppose $\chr\FF\ne 2,3$. An $\FF$-algebra with involution $(\cA,\invol)$ is said to be {\it structurable} if
\begin{equation}\label{struc}
[V_{x,y},V_{z,w}] = V_{V_{x,y}z,w}-V_{z,V_{y,x}w}\quad\mbox{for all}\;x,y,z\in\cA,
\end{equation}
where $ V_{x,y}(z) = \{x,y,z\}:= (x\bar y)z + (z\bar y)x - (z\bar x)y $.
\end{df}

We will always assume that $\cA$ is {\em unital}. In the case $\chr\FF\ne 2,3$, it is shown in \cite{A78} that identity $\eqref{struc}$ implies that $(\cA,\invol)$ is {\em skew-alternative}, i.e.,
\[
(z-\bar{z},x,y)=-(x,z-\bar{z},y)=(x,y,z-\bar{z})\quad\mbox{for all}\;x,y,z\in\cA,
\]
where $(a,b,c)\bydef(ab)c-a(bc)$. In the case $\chr\FF=2$ or $3$, skew-alternativity is taken as an additional axiom.

Denote by $Z(\cA)$ the associative center of $\cA$ (i.e., the set of elements $z\in\cA$ satisfying $xz=zx$ and $(z,x,y)=(x,z,y)=(x,y,z)=0$ for all $x,y\in\cA$). The {\it center} of $(\cA,\invol)$ is defined by $Z(\cA,\invol)=Z(\cA)\cap\sym(\cA,\invol)$. A (unital) structurable algebra $\cA$ is said to be {\it central} if $Z(\cA,\invol)=\FF1$.

\begin{theorem}[Allison, Smirnov] If $\chr\FF \neq 2, 3, 5$, then any central simple structurable $\FF$-algebra belongs to one of the of the following six (non-disjoint) classes:

\begin{itemize}
\item[(1)] central simple associative algebras with involution,
\item[(2)] central simple Jordan algebras (with identity involution),
\item[(3)] structurable algebras constructed from a non-degenerate Hermitian form
over a central simple associative algebra with involution,
\item[(4)] forms of the tensor product of two composition algebras,
\item[(5)] simple structurable algebras of skew-dimension 1 (forms of structurable matrix algebras),
\item[(6)] an exceptional 35-dimensional case (Kantor-Smirnov algebra), which can be constructed from an
octonion algebra.
\end{itemize}
\end{theorem}
The classification was given by Allison in the case of characteristic $0$ (see \cite{A78}), but case (6) was overlooked. Later, Smirnov completed the classification and gave the generalization for $\chr\FF\neq 2,3,5$ (see \cite{Smi92}).

\subsection{Structurable matrix algebras}

Assume $\chr\FF\ne 2,3$. Let $J$ and $J'$ be vector spaces over $\FF$ and consider a triple $(T,N,N')$ where $N$ and $N'$ are symmetric trilinear forms on $J$ and $J'$, respectively, and $T\colon J\times J'\rightarrow\FF$ is a nondegenerate bilinear form. For any $x,y\in J$, $x',y'\in J'$, define $x\times y\in J'$ and $x'\times y'\in J$ by
\begin{equation*}
T(z,x\times y)=N(x,y,z) \quad \text{and} \quad T(x'\times y',z')=N'(x',y',z')
\end{equation*}
for all $z\in J$, $z'\in J'$. For any $x\in J$ and $x'\in J'$, define $N(x)=\frac16N(x,x,x)$, $N'(x')=\frac16N'(x',x',x')$, $x^\#=\frac12\,x\times x$ and $x'^\#=\frac12\,x'\times x'$. If the triple $(T,N,N')$ satisfies the identities
\begin{equation*}
(x^\#)^\#=N(x)x \quad \text{and} \quad (x'^\#)^\#=N'(x')x'
\end{equation*}
for all $x\in J$, $x'\in J'$, then the algebra
\[ \cA=\left\{\begin{pmatrix}\alpha & x \\ x' & \beta\end{pmatrix}\;|\;\alpha,\beta\in\FF, x\in J, x'\in J' \right\}, \]
with multiplication
\begin{equation} \label{matrstructproduct}
\begin{pmatrix} \alpha & x \\ x' & \beta \end{pmatrix}
\begin{pmatrix} \gamma & y \\ y' & \delta \end{pmatrix} =
\begin{pmatrix} \alpha\gamma+T(x,y') & \alpha y+\delta x+x'\times y' \\
		\gamma x'+\beta y'+x\times y & T(y,x')+\beta\delta \end{pmatrix},
\end{equation}
and involution
\begin{equation} \label{matrstructinvol}
\begin{pmatrix} \alpha & x \\ x' & \beta \end{pmatrix} \stackrel{-}{\mapsto}
\begin{pmatrix} \beta & x \\ x' & \alpha \end{pmatrix},
\end{equation}
is a central simple structurable algebra of skew-dimension $1$, where the space of skew elements is spanned by $s_0=\matr{1&0\\0&-1}$. These are called {\it structurable matrix algebras} in \cite{AF84}, where it is shown (see Proposition 4.5) that, conversely, if $(\cA,\invol)$ is a simple structurable algebra with $\sks(\cA,\invol)=\FF s_0\ne 0$, then $s_0^2=\mu 1$ with $\mu\in\FF^\times$, and $(\cA,\invol)$ is isomorphic to a structurable matrix algebra if and only if $\mu$ is a square in $\FF$.

The triples $(T,N,N')$, as above, that satisfy $N\ne 0$ (equivalently, $N'\ne 0$) are called {\em admissible triples} in \cite{A78}, where it is noted that the corresponding structurable algebras possess a nondegenerate symmetric bilinear form
\begin{equation}\label{df:trace_form_A}
\langle a,b\rangle=\tr(a\bar{b}),\quad\mbox{where}\quad\tr\begin{pmatrix} \alpha & x \\ x' & \beta \end{pmatrix}\bydef\alpha+\beta,
\end{equation}
which is {\em invariant} in the sense that $\langle \bar{a},\bar{b}\rangle=\langle a,b\rangle$ and $\langle ca,b\rangle=\langle a,\bar{c}b\rangle$ for all $a,b,c$.
The main source of admissible triples are Jordan algebras: if $J$ is a separable Jordan algebra of degree $3$ with generic norm $N$ and generic trace $T$, then $(\zeta T,\zeta N,\zeta^2 N)$ is an admissible triple (with $J'=J$) for any nonzero $\zeta\in\FF$.
Note that the map $x\mapsto \lambda x$ and $x'\mapsto\lambda^2 x'$ is an isomorphism from $(\lambda^3 T,\lambda^3 N,\lambda^6 N)$ to $(T,N,N)$, so over algebraically closed fields, we can get rid of $\zeta$.

\subsection{Cayley--Dickson doubling process for algebras with involution}

Let $(\cB,\invol)$ be a unital $\FF$-algebra with involution,  $\chr\FF\ne 2$, and let $\phi\colon\cB\times\cB\rightarrow\FF$ be a symmetric bilinear form such that $\phi(1,1)\neq 0$ and $\phi(b,1)=\phi(\bar b,1)$ for all $b\in\cB$. Denote $\phi(b)=\phi(b,1)$ and define $\theta\colon\cB\rightarrow\cB$ by
\begin{equation*}
b^\theta = -b + \frac{2\phi(b)}{\phi(1)}1.
\end{equation*}
Then $\theta$ is a linear map that commutes with the involution and satisfies $\theta^2=\id$ and $\phi(b_1^\theta,b_2^\theta)=\phi(b_1,b_2)$ for all $b_1,b_2\in\cB$. Given $0\neq\mu\in\FF$, define a new algebra with involution $\CD(\cB,\mu):=\cB\oplus\cB$ where multiplication is given by
\begin{equation}\label{df:mult_of_A}
(b_1,b_2)(c_1,c_2)=(b_1c_1+\mu(b_2c_2^\theta)^\theta, b_1^\theta c_2+(b_2^\theta c_1^\theta)^\theta)
\end{equation}
and involution is given by
\begin{equation*}
\overline{(b_1,b_2)} = (\bar b_1, -(\bar{b}_2)^\theta).
\end{equation*}
Note that $b\in\cB$ can be identified with $(b,0)$, that $(0,b)=vb$ for $v\bydef(0,1)$, and $v^2=\mu 1$. Thus $(b_1,b_2)=b_1+vb_2$ and $\CD(\cB,\mu)=\cB\oplus v\cB$. Moreover, the symmetric bilinear form $\phi$ can be extended to $\cB\oplus v\cB$ by setting  $\phi(b_1+vb_2,c_1+vc_2)=\phi(b_1,c_1)-\mu\phi(b_2,c_2)$; the extended $\phi$ satisfies $\phi(1,1)\neq 0$ and $\phi(a,1)=\phi(\bar{a},1)$ for all $a\in\CD(\cB,\mu)$.

This construction was introduced in \cite{AF84} and called the {\em (generalized) Cayley--Dickson process} because it reduces to the classical doubling process for a Hurwitz algebra $\cB$ if $\phi$ is the polar form of the norm and hence $b^\theta=\bar{b}$ for all $b\in\cB$.

It is shown in \cite{AF84} assuming $\chr\FF\ne 2,3$ (see Theorem 6.6, where a slightly more general situation is considered) that if $\cB$ is a separable Jordan algebra of degree $4$, the involution is trivial and $\phi$ is the generic trace form, then $\CD(\cB,\mu)$ is a simple structurable algebra of skew-dimension $1$. In fact, if $\mu$ is a square in $\FF$ then such $\CD(\cB,\mu)$ is isomorphic to the structurable matrix algebra corresponding to a certain admissible triple defined on the space $\cB_0\subset\cB$ of elements with generic trace $0$ (Proposition 6.5).

So let $\cB$ be a separable Jordan algebra of degree $4$ and let $\cA=\CD(\cB,\mu)$ as above. We state some basic properties of $\cA$ for future use: $\cB$ is a subalgebra of $\cA$, there is an element $v\in\cB$ such that $\cA=\cB\oplus v\cB$, and the involution of $\cA$ is given by $\overline{a+vb}=a-vb^\theta$ where $\theta\colon\cB\to\cB$ is a linear map defined by $1^\theta=1$ and $b^\theta=-b$ for all $b\in\cB_0$. The operators $L_v$ and $R_v$ of left and right multiplication by $v$, respectively, satisfy the relations $L_v^2=R_v^2=\mu\id$ and $L_v R_v=R_v L_v=\mu\theta$ where we extended $\theta$ to an operator on $\cA$ by the rule $(a+vb)^\theta=a^\theta+vb^\theta$. The multiplication of $\cA$ is determined by the formulas
\begin{equation}\label{eq:mult_of_A}
a(vb)=v(a^\theta b),\;(va)b=v(a^\theta b^\theta)^\theta,\;(va)(vb)=\mu(ab^\theta)^\theta,
\end{equation}
for all $a,b\in\cA$. (This is equivalent to \eqref{df:mult_of_A} if $a,b\in\cB$, but a straightworward computation shows that the formulas continue to hold if we allow $a$ and $b$ to range over $\cA$.)

Since $\sks(\cA,\invol)=\FF v$ and $v^2=\mu 1$, all automorphisms of $(\cA,\invol)$ send $v$ to $\pm v$ and all derivations of $(\cA,\invol)$ annihilate $v$. Every automorphism (or derivation) $\vphi$ of $\cB$ extends to $\cA$ in the natural way: $a+vb\mapsto \vphi(a)+v\vphi(b)$. We will denote this extended map by the same symbol. Similarly, any $G$-grading $\cB=\bigoplus_{g\in G}\cB_g$ gives rise to a $G$-grading on $\cA$, namely, $\cA=\bigoplus_{g\in G}(\cB_g\oplus v\cB_g)$.

\subsection{Brown algebras via Cayley--Dickson process}\label{ss:Brown_via_CD}

Assume $\chr\FF\ne 2$. The split Brown algebra mentioned in the introduction can be obtained as the Cayley--Dickson double of two different separable Jordan algebras of degree $4$. We will consider a more general situation.

Let $\cQ$ be a quaternion algebra over $\FF$ with its standard involution, $q\mapsto\bar{q}$. The algebra $M_4(\cQ)$ is associative and has a natural involution $(q_{ij})^*=(\wb{q}_{ji})$, so $\HQ\bydef\{x\in M_4(\cQ)\;|\;x^*=x\}$ is a Jordan algebra with respect to the symmetrized product $(x,y)\mapsto\frac12(xy+yx)$. This is a simple Jordan algebra of degree $4$ and dimension $28$, so $\CD(\HQ,\mu)$ is a structurable algebra of dimension $56$, for any $\mu\in\FF^\times$.

\begin{remark}
If $\chr\FF=3$, we cannot apply the results in \cite{AF84} directly, but $\cQ$ can be obtained by ``extension of scalars'' from the ``generic'' quaternion algebra $\tilde\cQ$ over the polynomial ring $\ZZ[X,Y]$, hence $\HQ$ can be obtained from the Jordan algebra $\sym_4(\tilde\cQ)$ over $\ZZ[\frac12][X,Y]$, and $\CD(\HQ,\mu)$ can be obtained from the algebra $\CD(\sym_4(\tilde{\cQ}),Z)$ over $\ZZ[\frac12][X,Y,Z]$, which satisfies the required identities because it is a subring with involution in a structurable algebra over the field $\QQ(X,Y,Z)$.
\end{remark}

Let $\cC$ be an octonion algebra over $\FF$. As in the case of quaternions, the standard involution of $\cC$ yields an involution on the (nonassociative) algebra $M_3(\cC)$. It is well known that $\alb=\sym_3(\cC)$ (Albert algebra) is an exceptional Jordan algebra of dimension $27$. We will use the standard notation:
\[
\alb=\FF E_1\oplus\FF E_2\oplus\FF E_3\oplus \iota_1(\cC)\oplus\iota_2(\cC)\oplus\iota_3(\cC),
\]
where
\[
\begin{aligned}
E_1&=\begin{pmatrix}1&0&0\\ 0&0&0\\ 0&0&0\end{pmatrix}, &
E_2&=\begin{pmatrix}0&0&0\\ 0&1&0\\ 0&0&0\end{pmatrix}, &
E_3&=\begin{pmatrix}0&0&0\\ 0&0&0\\ 0&0&1\end{pmatrix}, \\
\iota_1(a)&=2\begin{pmatrix}0&0&0\\ 0&0&\bar a\\ 0&a&0\end{pmatrix},\quad &
\iota_2(a)&=2\begin{pmatrix}0&0&a\\ 0&0&0\\ \bar a&0&0\end{pmatrix},\quad &
\iota_3(a)&=2\begin{pmatrix}0&\bar a&0\\a&0&0\\ 0&0&0\end{pmatrix},\quad
\end{aligned}
\]
for any $a\in\cC$. Then $E_i$ are orthogonal idempotents with $E_1+E_2+E_3=1$, and the remaining products are as follows:
\begin{equation*}
\begin{split}
&E_i\iota_i(a)=0,\quad E_{i+1}\iota_i(a)=\frac{1}{2}\iota_i(a)=E_{i+2}\iota_i(a),\\
&\iota_i(a)\iota_{i+1}(b)=\iota_{i+2}(\bar a\bar b),\quad
\iota_i(a)\iota_i(b)=2n(a,b)(E_{i+1}+E_{i+2}),
\end{split}
\end{equation*}
for all $a,b\in \cC$, with $i=1,2,3$ taken modulo $3$. (This convention about indices will be used without further mention.)

The Jordan algebra $\alb$ is simple. Any element $x\in\alb$ satisfies the generic degree $3$ equation
$
x^3-T(x)x^2+S(x)x-N(x)1=0,
$
for the linear form $T$ (the generic trace), the quadratic form $S$, and the cubic form $N$ (the generic norm) given by:
\begin{equation*}
\begin{split}
T(x)&=\alpha_1+\alpha_2+\alpha_3,\\
S(x)&=\frac{1}{2}\bigl(T(x)^2-T(x^2)\bigr)=\sum_{i=1}^3 \bigl(\alpha_{i+1}\alpha_{i+2}-4n(a_i)\bigr),\\
N(x)&=\alpha_1\alpha_2\alpha_3+8n(a_1,\bar a_2\bar a_3)-4\sum_{i=1}^3\alpha_in(a_i),
\end{split}
\end{equation*}
for $x=\sum_{i=1}^3\bigl(\alpha_iE_i+\iota_i(a_i)\bigr)$, where $n$ is the norm of $\cC$.

Hence $\alb\times\FF$ is a separable Jordan algebra of degree $4$ and dimension $28$, so $\CD(\alb\times\FF,\mu)$ is a structurable algebra of dimension $56$, for any $\mu\in\FF^\times$.

The connection between the above two Cayley--Dickson doubles is the following: if $\cC=\CD(\cQ,\mu)$, then $\CD(\HQ,\mu)$ is isomorphic to $\CD(\alb\times\FF,\mu)$. Indeed, we have $\CD(\HQ,\mu)=\HQ\oplus v\HQ$ and $\CD(\alb\times\FF,\mu)=(\alb\times\FF)\oplus v'(\alb\times\FF)$ with $v^2=\mu 1=v'^2$. For any $a\in\cQ$, define the elements of $\HQ$:
\begin{equation*}
\iota_1'(a)= \begin{pmatrix} 0&0&0&2a\\0&0&0&0\\0&0&0&0\\2\bar a&0&0&0 \end{pmatrix}, \;
\iota_2'(a)= \begin{pmatrix} 0&0&0&0\\0&0&0&2a\\0&0&0&0\\0&2\bar a&0&0 \end{pmatrix}, \;
\iota_3'(a)= \begin{pmatrix} 0&0&0&0\\0&0&0&0\\0&0&0&2a\\0&0&2\bar a&0 \end{pmatrix}.
\end{equation*}
Then we have a $\ZZ_2$-grading on $\HQ$ given by $\HQ_{\bar0}=\diag(\sym_3(\cQ),\FF)$ and $\HQ_{\bar1}=\bigoplus_{j=1}^3\iota'_j(\cQ)$. The automorphism of order 2 producing this grading can be extended to an automorphism  of $\cA=\CD(\HQ,\mu)$ sending $v$ to $-v$, which also has order $2$ and will be denoted by $\Upsilon$. The fixed subalgebra of $\Upsilon$ is $\cB=\diag(\sym_3(\cQ),\FF) \oplus \bigoplus_{j=1}^3 v\iota'_j(\cQ)$. The involution is trivial on $\cB$, so it is a Jordan algebra. Since $L_v$ is an invertible operator, the $\ZZ_2$-grading produced by $\Upsilon$ is $\cA=\cB\oplus v\cB$. Write $\cC=\cQ\oplus u\cQ$ with $u^2=\mu 1$. Then it is straightforward to verify that the mapping $\varphi_{\CD}\colon\cB \rightarrow \alb\times\FF$ defined by  $\diag(x,\lambda)\mapsto(x,\lambda)$, for $x\in\sym_3(\cQ)$, $\lambda\in\FF$, and $v\iota'_j(a) \mapsto (\iota_j(ua),0)$, for $a\in\cQ$, is an isomorphism of algebras. Moreover, we have $\vphi_{\CD}(b^\theta)=\vphi_{\CD}(b)^\theta$ for all $b\in\cB$, so identities \eqref{eq:mult_of_A} for the algebra $\cA$ imply that $\vphi_{\CD}$ can be extended to an isomorphism $\varphi_{\CD}\colon\CD(\HQ,\mu) \rightarrow \CD(\alb\times\FF,\mu)$ sending $v$ to $v'$.

\begin{df}\label{df:Brown}
Let $\cQ$ be a quaternion algebra over $\FF$ and let $\cC=\CD(\cQ,1)$, so $\cC$ is the split octonion algebra and $\alb=\sym_3(\cC)$ is the split Albert algebra. Then the structurable algebra $\CD(\sym_4(\cQ),1)\cong\CD(\alb\times\FF,1)$ will be referred to as the {\em split Brown algebra}.
\end{df}

\subsection{Brown algebras as structurable matrix algebras}\label{ss:Brown_via_matr}

It is shown in \cite{AF84}, assuming $\chr\FF\ne 2,3$, that the admissible triple $(T,N,N)$ arising from a separable Jordan algebra $\cJ$ of degree $3$ can be realized on the space of elements with generic trace $0$ in the separable Jordan algebra $\cJ\times\FF$ of degree $4$ (see Propositions 5.6 and 6.5) so that $\CD(\cJ\times\FF,1)$ is isomorphic to the structurable matrix algebra defined by $(T,N,N)$. We will now exhibit this isomorphism for the case $\cJ=\alb$ and see that it also works in the case $\chr\FF=3$.

\begin{remark}
If $\chr\FF=3$, we can still define ``structurable matrix algebras'' starting from the cubic form $N(x)$ and taking its polarization for the symmetric trilinear form $N(x,y,z)$.
\end{remark}

For the admissible triple $(T,N,N)$ on $\alb$, we have $x^\#:=x^2-T(x)x+S(x)1$ ({\em Freudenthal adjoint}),  $x\times y=(x+y)^\#-x^\#-y^\#$ ({\em Freudenthal cross product}) and $S(x)=\frac{1}{2}(T(x)^2-T(x^2))$ for any $x,y\in\alb$, hence we have the identities
\begin{equation}\label{eq:cross_product}
x\times x = 2x^2 - 2T(x)x + \bigl(T(x)^{2}- T(x^2)\bigr)1 \quad \text{and} \quad  x\times 1 = T(x)1-x.
\end{equation}
Let $\tilde{\cA}$ be the corresponding structurable matrix algebra and let $s_0=\matr{1 & 0 \\ 0 & -1}$, so $s_0$ spans the space of skew elements and $s_0^2=1$. For $x\in\alb$, denote $\eta(x)= \matr{0 & x \\ 0 & 0}$ and $\eta'(x)= \matr{0 & 0 \\ x & 0}$. The subalgebra $\tilde{\cB}:=\{\eta(x)+\eta'(x)+\lambda 1\;|\;x\in\alb,\lambda\in\FF\}$ of $\tilde{\cA}$ consists of symmetric elements, so it is a Jordan algebra. We claim that it is isomorphic to $\alb\times\FF$. Indeed, define a linear injection $\iota\colon\alb\rightarrow\tilde{\cB}$ by setting  $\iota(x)=\frac{1}{4}\bigl(\eta(2x-T(x)1) + \eta'(2x-T(x)1) + T(x)1\bigr)$ for all $x\in\alb$. Using identities \eqref{eq:cross_product}, one verifies that $\iota(x)^2=\iota(x^2)$, so $\iota$ is a nonunital monomorphism of algebras. Then $e_\alb=\iota(1)$ and $e_\FF=1-e_\alb$ are orthogonal idempotents and $\tilde{\cB}=\iota(\alb)\oplus \FF e_\FF$. We conclude that $\alb\times\FF \to \tilde{\cB}$, $(x,\lambda)\mapsto\iota(x)+\lambda e_\FF$, is an isomorphism of algebras. This isomorphism extends to an isomorphism $\CD(\alb\times\FF,1)=(\alb\times\FF)\oplus v'(\alb\times\FF)\rightarrow\tilde{\cA}$ sending $v'$ to $s_0$.

\section{A construction in terms of the double of $\HQ$}\label{s:H4}

Let $\cQ$ be the split quaternion algebra over a field $\FF$, $\chr{\FF}\ne 2$, i.e., $\cQ\cong M_2(\FF)$ and the standard involution switches $E_{11}$ with $E_{22}$ and multiplies both $E_{12}$ and $E_{21}$ by $-1$. The subalgebra $\cK=\lspan{E_{11},E_{22}}$ is isomorphic to $\FF\times\FF$ with exchange involution.

Consider the associative algebra $M_4(\cQ)$ with involution $(q_{ij})^*=(\wb{q}_{ji})$. Since $M_4(\cQ)\cong M_4(\FF)\ot\cQ$, we can alternatively write the elements of $M_4(\cQ)$ as sums of tensor products or as $2\times 2$ matrices over $M_4(\FF)$. The involution on $M_4(\cQ)$ is the tensor product of matrix transpose $x\mapsto x^t$ on $M_4(\FF)$ and the standard involution on $\cQ$. Consider the Jordan subalgebra of symmetric elements
\[
\HQ=\{a\in M_4(\cQ)\;|\;a^*=a\}=\left\{\begin{pmatrix}z&x\\y&z^t\end{pmatrix}\;|\;x,y,z\in M_4(\FF),\,x^t=-x,\,y^t=-y\right\}.
\]
Note that the subalgebra $\HK\subset\HQ$ is isomorphic to $\Jord{M_4(\FF)}$.

Let $\cJ=\HQ$ and define $\cA=\CD(\cJ,1)=\cJ\oplus v\cJ$ as in Subsection \ref{ss:Brown_via_CD}. We want to construct a $\ZZ_4^3$-grading on $\cA$ assuming $\FF$ contains a 4-th root of unity $\bi$. The construction will proceed in two steps: first we define a $\ZZ_4$-grading on $\cA$ and then refine it using two commuting automorphisms of order $4$. Hence, the subalgebra $\HK\oplus v\HK$ will carry a $\ZZ_2\times\ZZ_4^2$-grading. The elements of $\ZZ_4$ will be written as integers with a bar.

The even components of the $\ZZ_4$-grading are just $\cA_\zero=\HK$ and $\cA_\two=v\HK$. The odd components are as follows:
\begin{align*}
\cA_\one&=\{x\ot E_{12}+v(y\ot E_{21})\;|\;x,y\in M_4(\FF),\,x^t=-x,\,y^t=-y\}\quad\mbox{and}\\
\cA_\three&=\{x\ot E_{21}+v(y\ot E_{12})\;|\;x,y\in M_4(\FF),\,x^t=-x,\,y^t=-y\}=v\cA_\one.
\end{align*}
It is straightforward to verify that $\cA=\cA_\zero\oplus\cA_\one\oplus\cA_\two\oplus\cA_\three$ is indeed a $\ZZ_4$-grading. Moreover, the coarsening induced by the quotient map $\ZZ_4\to\ZZ_2$ is the $\ZZ_2$-grading obtained by extending the standard $\ZZ_2$-grading of $\cQ=M_2(\FF)$.

The algebra $M_4(\FF)$ has a $\ZZ_4^2$-grading associated to the generalized Pauli matrices:
\[
X=\begin{bmatrix}1&0&0&0\\0&\bi&0&0\\0&0&-1&0\\0&0&0&-\bi\end{bmatrix}\quad\mbox{and}\quad
Y=\begin{bmatrix}0&1&0&0\\0&0&1&0\\0&0&0&1\\1&0&0&0\end{bmatrix},
\]
namely, the component of degree $(\bar{k},\bar{\ell})$ is $\FF X^kY^\ell$. This grading is the eigenspace decomposition with respect to the commuting order $4$ automorphisms $\Ad X$ and $\Ad Y$. Note that the group $\GL_4(\FF)$ acts on $\sym_4(\cQ)$ via $g\mapsto\Ad\matr{g&0\\0&(g^t)^{-1}}$. Indeed, this matrix is unitary with respect to our involution on $M_4(\cQ)$ and hence the conjugation leaves the space $\HQ$ invariant. Explicitly, the action on $\HQ$ is the following:
\begin{equation}\label{eq:GL4_action}
g\cdot\begin{pmatrix}z&x\\y&z^t\end{pmatrix}=\begin{pmatrix}gzg^{-1}&gxg^t\\(g^{-1})^tyg^{-1}&(gzg^{-1})^t\end{pmatrix}
\quad\mbox{for all }x,y,z\in M_4(\FF),\,x^t=-x,\,y^t=-y.
\end{equation}
Substituting $X$ and $Y$ for $g$, we obtain two order $4$ automorphisms of $\HQ$, which will be denoted by $\vphi$ and $\psi$, respectively. Observe, however, that $\vphi$ and $\psi$ do not commute: their commutator is the identity on the even component of $\HQ$, relative to the standard $\ZZ_2$-grading of $\cQ$, and the negative identity on the odd component. In fact, the classification of gradings on $\HQ$ is the same as on the (split) Lie algebra of type $C_4$ (since our involution on $M_4(\cQ)\cong M_8(\FF)$ is symplectic), and the latter algebra does not admit a group grading whose support generates $\ZZ_4^2$. The good news is that the extensions of $\vphi$ and $\psi$ (which we denote by the same letters) preserve the $\ZZ_4$-grading of $\cA$, so each of them can be used to refine it to a $\ZZ_4^2$-grading.

To resolve the difficulty described above, we are going to construct another order $4$ automorphism $\pi$ of $\cA$ preserving the $\ZZ_4$-grading and use $\pi$ to make a correction to $\vphi$. We want $\pi$ to be the identity on the even component $\cA_\zero\oplus\cA_\two$ and switch around the terms containing $E_{12}$ and $E_{21}$ for the elements in the odd component $\cA_\one\oplus\cA_\three$. Observe that the spaces $U=\{x\ot E_{12}\;|\;x^t=-x\}$ and $V=\{y\ot E_{21}\;|\;y^t=-y\}$ are dual $\GL_4(\FF)$-modules with respect to the action \eqref{eq:GL4_action}. Formally, their duality can be established via the invariant nondegenerate pairing $(x\ot E_{12},y\ot E_{21})=-\frac12\tr(xy)$, which is a scaling of the restriction of the trace form of $M_8(\FF)$. Under this pairing, the bases of skew-symmetrized matrix units are dual to each other. Recall that, for any skew-symmetric matrix $x$ of size $2k$, we have $\det(x)=\pf(x)^2$ where $\pf(x)$, called Pfaffian, is a homogeneous polynomial of degree $k$ in the entries of $x$. An important property of Pfaffian is $\pf(gxg^t)=\det(g)\pf(x)$ for any $g$ and skew-symmetric $x$, so $\pf(x)$ is invariant under the action of $\SL_{2k}(\FF)$ on the space $\sks_{2k}(\FF)$ given by $g\cdot x=gxg^t$. For $k=2$, the Pfaffian is a nondegenerate quadratic form, namely,
\begin{equation*}
\pf(x)=x_{12}x_{34}-x_{13}x_{24}+x_{14}x_{23}\quad\mbox{for all }x\in\sks_4(\FF),
\end{equation*}
so it can be used to identify the $\SL_4(\FF)$-module $\sks_4(\FF)$ with its dual module. Identifying $U$  with $\sks_4(\FF)$ and $V$ with $\sks_4(\FF)^*$ as above, we obtain an $\SL_4(\FF)$-equivariant isomorphism $U\to V$. Using the basis of skew-symmetrized matrix units in $\sks_4(\FF)$, we immediately see that the isomorphism is given by $x\ot E_{12}\mapsto\wh{x}\ot E_{21}$ where
\begin{equation}\label{df:hat}
\mbox{if}\quad
x=\begin{bmatrix}0&\alpha&\beta&\gamma\\&0&\delta&\veps\\&&0&\zeta\\\mbox{skew}&&&0\end{bmatrix}\quad\mbox{then}\quad
\wh{x}=\begin{bmatrix}0&\zeta&-\veps&\delta\\&0&\gamma&-\beta\\&&0&\alpha\\\mbox{skew}&&&0\end{bmatrix}.
\end{equation}
By construction, we have $\wh{g\cdot x}=(g^{-1})^t\cdot\wh{x}$ for all $g\in\SL_4(\FF)$. This implies that, more generally,
\begin{equation}\label{eq:two_GL_actions}
\wh{gxg^t}=\det(g)(g^{-1})^t\wh{x}g^{-1}\quad\mbox{for all }x\in\sks_4(\FF)\mbox{ and }g\in\GL_4(\FF),
\end{equation}
and also, passing to the corresponding Lie algebra,
\begin{equation}\label{eq:two_gl_actions}
\wh{zx+xz^t}=\tr(z)\wh{x}-(z^t\wh{x}+\wh{x}z)\quad\mbox{for all }x\in\sks_4(\FF)\mbox{ and }z\in\Gl_4(\FF).
\end{equation}
Finally, we define $\pi\colon\cA\to\cA$ as identity on $\cA_\zero\oplus\cA_\two$ and
\begin{equation}\label{df:pi}
\begin{array}{ll}
\pi(x\ot E_{12})=-v(\wh{x}\ot E_{21}),& \pi(v(x\ot E_{12}))=-\wh{x}\ot E_{21},\\
\pi(x\ot E_{21})=\phantom{+}v(\wh{x}\ot E_{12}),& \pi(v(x\ot E_{21}))=\phantom{+}\wh{x}\ot E_{12}.
\end{array}
\end{equation}
Clearly, $\pi$ preserves the $\ZZ_4$-grading and $\pi^4=\id$.

\begin{lemma}
The map $\pi$ is an automorphism of $\cA$.
\end{lemma}

\begin{proof}
For $a_i=\matr{z_i&x_i\\y_i&z_i^t}\in\HQ$, $i=1,2$, we have, on the one hand,
\[
\begin{split}
\pi(a_1a_2)=&
\pi\begin{pmatrix}
\frac12(z_1z_2+z_2z_1)+\frac12(x_1y_2+x_2y_1)&\frac12(z_1x_2+x_2z_1^t)+\frac12(z_2x_1+x_1z_2^t)\\
\frac12(z_1^ty_2+y_2z_1)+\frac12(z_2^ty_1+y_1z_2)&\bigl(\frac12(z_1z_2+z_2z_1)+\frac12(x_1y_2+x_2y_1)\bigr)^t
\end{pmatrix}\\
=&\frac12\begin{pmatrix}
(z_1z_2+z_2z_1)+(x_1y_2+x_2y_1)&0\\
0&\bigl((z_1z_2+z_2z_1)+(x_1y_2+x_2y_1)\bigr)^t
\end{pmatrix}\\
&+\frac12v
\begin{pmatrix}
0&(\wh{z_1^ty_2+y_2z_1})+(\wh{z_2^ty_1+y_1z_2})\\
-(\wh{z_1x_2+x_2z_1^t})-(\wh{z_2x_1+x_1z_2^t})&0
\end{pmatrix}
\end{split}
\]
and, on the other hand,
\[
\begin{split}
\pi(a_1)\pi(a_2)=&
\left(
\begin{pmatrix}z_1&0\\0&z_1^t\end{pmatrix}+v\begin{pmatrix}0&\wh{y}_1\\-\wh{x}_1&0\end{pmatrix}
\right)\left(
\begin{pmatrix}z_2&0\\0&z_2^t\end{pmatrix}+v\begin{pmatrix}0&\wh{y}_2\\-\wh{x}_2&0\end{pmatrix}
\right)\\
=&\frac12\begin{pmatrix}
z_1z_2+z_2z_1&0\\
0&(z_1z_2+z_2z_1)^t
\end{pmatrix}
+\frac12\begin{pmatrix}
\wh{y}_1\wh{x}_2+\wh{y}_2\wh{x}_1&0\\
0&\wh{x}_1\wh{y}_2+\wh{x}_2\wh{y}_1
\end{pmatrix}^\theta
\\
&+\frac12v\left(
\tr(z_1)\begin{pmatrix}
0&\wh{y}_2\\
-\wh{x}_2&0
\end{pmatrix}
-\begin{pmatrix}
0&z_1\wh{y}_2+\wh{y}_2z_1^t\\
-z_1^t\wh{x}_2-\wh{x}_2z_1&0
\end{pmatrix}
\right)
\\
&+\frac12v\left(
\tr(z_2)\begin{pmatrix}
0&\wh{y}_1\\
-\wh{x}_1&0
\end{pmatrix}
-\begin{pmatrix}
0&z_2\wh{y}_1+\wh{y}_1z_2^t\\
-z_2^t\wh{x}_1-\wh{x}_1z_2&0
\end{pmatrix}\right).
\end{split}
\]
Examining the two expressions and taking into account identities \eqref{eq:two_gl_actions}, above, and \eqref{eq:trace_hat}, below, we see that $\pi(a_1a_2)=\pi(a_1)\pi(a_2)$. The identity
\begin{equation}\label{eq:trace_hat}
xy+\wh{y}\wh{x}=\frac12\tr(xy)1\quad\mbox{for all }x,y\in\sks_4(\FF)
\end{equation}
can be verified  by a direct computation, but it is a consequence of the fact that both $(x,y)\mapsto x\wh{y}+y\wh{x}$ and $(x,y)\mapsto\tr(x\wh{y})1$ are symmetric bilinear $\SL_4(\FF)$-equivariant maps $\sks_4(\FF)\times\sks_4(\FF)\to M_4(\FF)$, and the space of such maps has dimension $1$.

It remains to observe that $\pi$ commutes with the left multiplication by $v$ and with $\theta$ (hence also with the right multiplication by $v$), so we compute:
\begin{align*}
&\pi((va)b)=\pi(v(a^\theta b^\theta)^\theta)=v(\pi(a^\theta b^\theta))^\theta
=v(\pi(a)^\theta\pi(b)^\theta)^\theta=(v\pi(a))\pi(b)=\pi(va)\pi(b),\\
&\pi(a(vb))=\pi(v(a^\theta b))=v\pi(a^\theta b)=v(\pi(a)^\theta\pi(b))=\pi(a)(v\pi(b))=\pi(a)\pi(vb),\\
&\pi((va)(vb))=\pi((a b^\theta)^\theta)=\pi(a b^\theta)^\theta=(\pi(a)\pi(b)^\theta)^\theta=(v\pi(a))(v\pi(b))=\pi(va)\pi(vb),
\end{align*}
for all $a,b\in\HQ$, where we have made use of identities \eqref{eq:mult_of_A}.
\end{proof}

\begin{remark}
Identity \eqref{eq:trace_hat} is equivalent to $x\wh{x}=-\pf(x)1$, i.e., $\wh{x}$ is the negative of the so-called {\em Pfaffian adjoint} of $x$.
\end{remark}

Now that we have the automorphism $\pi$, it is easy to construct the $\ZZ_4^3$-grading on $\cA$. It follows from \eqref{eq:GL4_action}, \eqref{eq:two_GL_actions} and the fact $\det(X)=\det(Y)=-1$ that $\pi$ commutes with each of $\vphi$ and $\psi$ on $\cA_\zero\oplus\cA_\two$ and anticommutes on $\cA_\one\oplus\cA_\three$. We will keep $\psi$ and replace $\vphi$ by the composition of $\pi$ and the action of $\wt{X}=\diag(\omega,\omega^3,\omega^5,\omega^7)$ given by \eqref{eq:GL4_action}, where $\omega^2=\bi$. (We can temporarily extend $\FF$ if necessary so that it contains such an element.) We will denote this composition by $\wt{\vphi}$. Since $\wt{X}$ is a scalar multiple of $X$, its action still commutes with the action of $Y$ on $\cA_\zero\oplus\cA_\two$ and anticommutes on $\cA_\one\oplus\cA_\three$. But $\det(\wt{X})=1$, so the action of $\wt{X}$ commutes with $\pi$ everywhere. Therefore, $\wt{\vphi}$ and $\psi$ are commuting order $4$ automorphisms of $\cA$. Since they preserve the $\ZZ_4$-grading, taking the eigenspace decomposition of each component with respect to $\wt{\vphi}$ and $\psi$ is the desired $\ZZ_4^3$-grading of $\cA$.

Let us calculate the homogeneous elements. The matrices $X^k Y^\ell$ form a basis of $M_4(\FF)$ and are eigenvectors for $\Ad X$ with eigenvalues $(-\bi)^\ell$ and for $\Ad Y$  with eigenvalues $\bi^k$, where $k,\ell=0,1,2,3$. Taking the convention that
\[
\cA_{(\bar{j},\bar{k},\bar{\ell})}=\{a\in\cA_{\bar{j}}\;|\;\psi(a)=\bi^k,\,\wt{\vphi}(a)=(-\bi)^\ell\}
\]
and recalling that on $\cA_\zero\oplus\cA_\two$ the automorphisms $\wt{\vphi}$ and $\psi$ act as $X$ and $Y$, respectively, we see that
\begin{align*}
\cA_{(\zero,\bar{k},\bar{\ell})}&=\FF(X^kY^\ell\ot E_{11}+(X^kY^\ell)^t\ot E_{22});\\
\cA_{(\two,\bar{k},\bar{\ell})}&=\FF v(X^kY^\ell\ot E_{11}+(X^kY^\ell)^t\ot E_{22}).
\end{align*}
To find the homogeneous elements in $\cA_\one$ and $\cA_\three$, we will use the matrices
{\small
\[
\xi_{1,2}=\begin{bmatrix}0&1&0&\mp 1\\&0&\pm 1&0\\&&0&1\\\mbox{skew}&&&0\end{bmatrix},\;
\xi_{3,4}=\begin{bmatrix}0&1&0&\pm\bi\\&0&\pm\bi&0\\&&0&-1\\\mbox{skew}&&&0\end{bmatrix},\;
\xi_{5,6}=\begin{bmatrix}0&0&1&0\\&0&0&\pm\bi\\&&0&0\\\mbox{skew}&&&0\end{bmatrix}
\]
}
as a basis for $\sks_4(\FF)$. They are eigenvectors for the action of $Y$, with eigenvalues $\pm 1$, $\pm\bi$ and $\pm\bi$, respectively. Hence, the elements $\xi_i\ot E_{12}$ and $v(\xi_i\ot E_{21})$ of $\cA_\one$ have the same eigenvalues with respect to $\psi$. (We are using the fact $Y^t=Y^{-1}$.) Finally, the action of $\wt{\vphi}$ on $\cA_\one$ is given by
\[
\wt{\vphi}(x\ot E_{12}+v(y\ot E_{21}))=\wt{X}\wh{y}\wt{X}^t\ot E_{12}-v((\wt{X}^{-1})^t\wh{x}\wt{X}^{-1}\ot E_{21}).
\]
One checks using \eqref{df:hat} that, for $i=1,2,3,4$, the elements $\xi_i\ot E_{12}\pm\bi v(\xi_i\ot E_{21})$ are eigenvectors with respect to $\wt{\vphi}$. Their eigenvalues are $\mp\bi$ if $i=1$ or $2$, and $\pm\bi$ if $i=3$ or $4$. Also, for $i=5,6$, the elements $\xi_i\ot E_{12}\pm v(\xi_i\ot E_{21})$ are eigenvectors with respect to $\wt{\vphi}$, with eigenvalues $\pm 1$. Putting this information together, we obtain:
\begin{align*}
\cA_{(\one,\zero,\bar{\ell})}&=\FF(\xi_1\ot E_{12}+\bi^\ell v(\xi_1\ot E_{21})),\,\ell=1,3;\\
\cA_{(\one,\two,\bar{\ell})}&=\FF(\xi_2\ot E_{12}+\bi^\ell v(\xi_2\ot E_{21})),\,\ell=1,3;\\
\cA_{(\one,\one,\bar{\ell})}&=\FF(\xi_3\ot E_{12}-\bi^\ell v(\xi_3\ot E_{21})),\,\ell=1,3;\\
\cA_{(\one,\three,\bar{\ell})}&=\FF(\xi_4\ot E_{12}-\bi^\ell v(\xi_4\ot E_{21})),\,\ell=1,3;\\
\cA_{(\one,\one,\bar{\ell})}&=\FF(\xi_5\ot E_{12}+\bi^\ell v(\xi_5\ot E_{21})),\,\ell=0,2;\\
\cA_{(\one,\three,\bar{\ell})}&=\FF(\xi_6\ot E_{12}+\bi^\ell v(\xi_6\ot E_{21})),\,\ell=0,2.
\end{align*}
Since $v$ is homogeneous of degree $(\two,\zero,\zero)$, each component $\cA_{(\three,\bar{k},\bar{\ell})}$ can be obtained from $\cA_{(\one,\bar{k},\bar{\ell})}$ using left multiplication by $v$, which amount to switching $E_{12}$ and $E_{21}$.

Since all homogeneous components have dimension $1$, our $\ZZ_4^3$-grading on $\cA$ is fine. The support is a proper subset of $\ZZ_4^3$, which can be characterized as follows using the distinguished element $g_0=(\two,\zero,\zero)$ (the degree of $v$): an element $g$ does not belong to the support if and only if $2g=g_0$.

\section{A construction in terms of structurable matrix algebra}\label{s:MatrixStructAlgebra}

In this section, we will construct a $\ZZ_4^3$-grading for the model of the split Brown algebra as in Subsection \ref{ss:Brown_via_matr}, assuming $\FF$ contains a $4$-th root of unity $\bi$. Let $\cC$ be the split octonion algebra and let $\alb=\sym_3(\cC)$ be the split Albert algebra, with generic trace $T$ and generic norm $N$. Consider the structurable matrix algebra $\cA$ associated to the admissible triple $(T,N,N)$, i.e., the product is given by \eqref{matrstructproduct} and the involution is given by \eqref{matrstructinvol}. Note that the Freudenthal cross product on $\alb$, which appears in \eqref{matrstructproduct}, is given by:
\begin{itemize}
\item[i)] $E_i\times E_{i+1}=E_{i+2}$,\; $E_i\times E_i=0$,
\item[ii)] $E_i\times\iota_i(x)=-\iota_i(x)$,\; $E_i\times\iota_{i+1}(x)=0=E_i\times\iota_{i+2}(x)$,
\item[iii)] $\iota_i(x)\times\iota_i(y)=-4n(x,y)E_i$,\; $\iota_i(x)\times\iota_{i+1}(y)=2\iota_{i+2}(\bar x\bar y)$.
\end{itemize}

As shown in \cite{AM99}, for the $\ZZ_2^3$-grading on the split Cayley algebra $\cC$ one can choose a homogeneous basis $\{ x_g \mid g\in \ZZ_2^3 \}$ such that the product is given by $x_g x_h=\sigma(g,h)x_{g+h}$ where
\begin{align*}
& \sigma(g,h)=(-1)^{\psi(g,h)}, \\
& \psi(g,h)=h_1 g_2 g_3 + g_1 h_2 g_3 + g_1 g_2 h_3 +
\sum_{i\leq j} g_i h_j.
\end{align*}
Consider the para-Cayley algebra associated to $\cC$, i.e., the same vector space with the new product $x*y=\bar x\bar y$. Note that $x_g*x_h =\gamma(g,h)x_{g+h}$ where
\begin{align*}
& \gamma(g,h)=s(g)s(h)\sigma(g,h), \\
& s(g) = (-1)^{\phi(g)}, \\
& \phi(g) = \sum_ig_i + \sum_{i<j}g_ig_j +  g_1g_2g_3 ,
\end{align*}
because $s(g)=-1$ if $g\ne 0$ and $s(0)=1$, so $\bar x_g = s(g)x_g$ for all $g\in\ZZ_2^3$.

Denote $a_0=0$, $a_1=(\bar0,\bar1,\bar0)$, $a_2=(\bar1,\bar0,\bar0)$, $a_3=a_1+a_2$, $g_0=(\bar0,\bar0,\bar1)$ in $\ZZ_2^3$. We will consider the quaternion algebra $\cQ=\lspan{x_{a_i}}$ with the ordered basis $B_\cQ=\{x_{a_i} \;|\; i=0,1,2,3\}$, and $\cQ^\perp=\lspan{x_{g_0+a_i}}$ with the ordered basis $B_{\cQ^\perp}=\{x_{g_0+a_i}\;|\; i=0,1,2,3\}$. Thus, $B_\cC=B_\cQ\cup B_{\cQ^\perp}$ is an ordered basis of $\cC$. It will be convenient to write the values $\gamma(g,h)$ as an $8\times 8$ matrix according to this ordering and split this matrix into $4\times 4$ blocks: $\gamma=\matr{\gamma_{11} & \gamma_{12} \\ \gamma_{21} & \gamma_{22}}$, so $\gamma_{11}$ records the values for the support of $\cQ$, etc., and similarly for $\gamma(g,h)$.

A straightforward calculation shows that
\begin{gather} \label{gammamatrix}
\begin{aligned}
& \gamma_{11}=\begin{bmatrix} 1 & -1 & -1 & -1 \\ -1 & -1 & 1 & -1 \\ -1 & -1 & -1 & 1 \\ -1 & 1 & -1 & -1 \end{bmatrix}, \quad
& \gamma_{12}=\begin{bmatrix} -1 & -1 & -1 & -1 \\ -1 & 1 & -1 & 1 \\ -1 & 1 & 1 & -1 \\ -1 & -1 & 1 & 1 \end{bmatrix}, \\
& \gamma_{21}=\begin{bmatrix} -1 & 1 & 1 & 1 \\ -1 & -1 & -1 & 1 \\ -1 & 1 & -1 & -1 \\ -1 & -1 & 1 & -1 \end{bmatrix}, \quad
& \gamma_{22}=\begin{bmatrix} -1 & -1 & -1 & -1 \\ 1 & -1 & -1 & 1 \\ 1 & 1 & -1 & -1\\ 1 & -1 & 1 & -1 \end{bmatrix}.
\end{aligned}\end{gather}
Define $\sigma_j(h):=\sigma(a_j,g_0+h)$ for any $h\in\supp \cQ^\perp$, $j=1,2,3$. Note that the matrix of $\sigma(a_j,g_0+h)$, $h\in B_{\cQ^\perp}$, coincides with the matrix $\sigma_{11}$, which is given by
\begin{equation} \label{sigmamatrix}
\sigma_{11}=(\sigma(a_j,a_k))_{j,k}=\begin{bmatrix} 1 & 1 & 1 & 1 \\ 1 & -1 & 1 & -1 \\ 1 & -1 & -1 & 1 \\ 1 & 1 & -1 & -1 \end{bmatrix}.
\end{equation}

We will need the following result in our construction of the $\ZZ_4^3$-grading on $\cA$.

\begin{lemma} \label{lemmabasis}
The basis $B_\cC = B_\cQ \cup B_{\cQ^\perp}$ of $\cC$ has the following properties:
\begin{itemize}
\item[$(P_{11})$] \quad  $\gamma(g,g')=\gamma(g+a_j,g'+a_{j+1})$,
\item[$(P_{22})$] \quad $\gamma(h,h')=\sigma_j(h)\sigma_{j+1}(h')\gamma(h+a_j,h'+a_{j+1})$,
\item[$(P_{12})$] \quad  $\gamma(g,h)=\sigma_{j+1}(h)\sigma_{j+2}(g+h)\gamma(g+a_j,h+a_{j+1})$,
\item[$(P_{21})$] \quad  $\gamma(h,g)=\sigma_j(h)\sigma_{j+2}(g+h)\gamma(h+a_j,g+a_{j+1})$,
\end{itemize}
for all $g,g'\in \supp\cQ$, $h,h'\in \supp\cQ^\perp$ and $j\in\{1,2,3\}$.
\end{lemma}

\begin{proof}
To shorten the proof, we will use matrices, but we need to introduce some notation. For $j=1,2,3$, let $\sigma_j$ be the column of values $\sigma_j(h)$, $h\in\supp\cQ^\perp$, i.e., $\sigma_j$ is the traspose of the corresponding row of matrix $\sigma_{11}$. We will denote by $\cdot$ the entry-wise product of matrices. (It is interesting to note that the rows and columns of $\sigma_{11}$ are the characters of $\ZZ_2^2$, which is related to the obvious fact $\sigma_j\cdot\sigma_{j+1}=\sigma_{j+2}$.) Denote $M_{\sigma_j}=[\sigma_j|\sigma_j|\sigma_j|\sigma_j]$ (the column $\sigma_j$ repeated $4$ times), and define the permutation matrices
\begin{equation*}
P_1=\left[\begin{smallmatrix} 0&1&0&0 \\ 1&0&0&0 \\ 0&0&0&1 \\ 0&0&1&0 \end{smallmatrix}\right] \quad\mbox{and}\quad
P_2=\left[\begin{smallmatrix} 0&0&1&0 \\ 0&0&0&1 \\ 1&0&0&0 \\ 0&1&0&0 \end{smallmatrix}\right] .
\end{equation*}

Note that the properties asserted in this lemma possess a cyclic symmetry in $j=1,2,3$ (as can be checked in the four blocks of $\gamma$), so it suffices to verify them for $j=1$. Then, property $(P_{11})$ can be written as $\gamma_{11}=P_1\gamma_{11}P_2$, because $P_1\gamma_{11}P_2$ is the matrix associated to $\gamma(g+a_1,g'+a_2)$. Similarly, property $(P_{22})$ can be writen as $\gamma_{22}=M_{\sigma_1}\cdot(P_1\gamma_{22}P_2)\cdot M_{\sigma_2}^t$.
Note that $\sigma$ is multiplicative in the second variable (because $\psi$ is linear in the second variable), so $\sigma_3(g+h)=\sigma(a_3,g+g_0+h)=\sigma(a_3,g)\sigma(a_3,g_0+h)=\sigma_3(g_0+g)\sigma_3(h)$. Therefore, $(P_{12})$ and $(P_{21})$ can be written as  $\gamma_{12}=M_{\sigma_3}\cdot(P_1\gamma_{12}P_2)\cdot M^t_{\sigma_2\cdot\sigma_3}$ and $\gamma_{21}=M_{\sigma_1\cdot\sigma_3}\cdot(P_1\gamma_{21}P_2)\cdot M^t_{\sigma_3}$. It is straightforward to check these four matrix equations.
\end{proof}

We will consider $\ZZ_2^3$ as a subgroup of $\ZZ_4^3$ via the embedding $a_1\mapsto (\bar0,\bar2,\bar0)$, $a_2\mapsto (\bar0,\bar0,\bar2)$, $a_3\mapsto (\bar0,\bar2,\bar2)$ and $g_0\mapsto (\bar2,\bar0,\bar0)$, so we can assume that $\gamma$ and $\sigma$ are defined on a subgroup of $\ZZ_4^3$ and take values as recorded in matrices \eqref{gammamatrix} and \eqref{sigmamatrix}. Define $b_1=(\bar0,\bar1,\bar0)$, $b_2=(\bar0,\bar0,\bar1)$ and $b_3=-b_1-b_2$ in $\ZZ_4^3$. Note that $\sum b_j=0$ and $a_j\mapsto 2b_j$ under the embedding.

Now we will define a $\ZZ_4^3$-grading on $\cA$ by specifying a homogeneous basis. For each $g\in\supp \cQ$, $h\in\supp \cQ^\perp$ and $j\in\{1,2,3\}$, consider the elements of $\cA$:
\begin{align*}
\alpha_{j,g}&:= \begin{pmatrix} 0 & \iota_j(x_g) \\ \iota_j(x_{g+a_j}) & 0 \end{pmatrix}, &
 \alpha'_{j,h}&:= \begin{pmatrix} 0 & \sigma_j(h)\bi \iota_j(x_h) \\ \iota_j(x_{h+a_j}) & 0 \end{pmatrix}, \\
\veps_j&:= \begin{pmatrix} 0 & E_j \\ E_j & 0 \end{pmatrix}, &
 \veps'_j&:=\veps_js_0= \begin{pmatrix} 0 & -E_j \\ E_j & 0 \end{pmatrix}.
\end{align*}
Then $B_\cA=\{1, s_0, \alpha_{j,g}, \alpha_{j,g}s_0, \alpha'_{j,h}, \alpha'_{j,h}s_0,\veps_j,\veps'_j\}$ is a basis of $\cA$. Set
\begin{equation}\label{df:Diego_grad}
\begin{array}{ll}
\deg(1)\bydef 0, & \deg(\veps_j)\bydef a_j,\\
\deg(\alpha_{j,g})\bydef b_j + g, & \deg(\alpha'_{j,h})\bydef (\one,\zero,\zero) + b_j + h,\\
\deg(xs_0)\bydef\deg(x)+g_0 & \mbox{for}\;x\in\{1,\alpha_{j,g},\alpha'_{j,h},\veps_j\}.
\end{array}
\end{equation}

To check that \eqref{df:Diego_grad} defines a $\ZZ_4^3$-grading, we compute the products of basis elements.

\begin{proposition} \label{constants}
For any elements $x,y\in B_\cA\setminus\{1,s_0\}$, $xy=yx$ if $\deg(x)+\deg(y)\neq g_0$ and $xy=-yx$ otherwise. The products of the elements of $B_\cA$ are then determined as follows:
\begin{itemize}
\item[i)] $s_0^2=\veps_j^2=1=-\veps'^2_j$, $\veps_j\veps_{j+1}=\veps_{j+2}$,
\item[ii)] $\veps_j\veps'_j=s_0$, $\veps'_j\veps'_{j+1}=\veps_{j+2}$, $\veps_j\veps'_{j+1}=-\veps'_{j+2}$, $\veps_{j+1}\veps'_j=-\veps'_{j+2}$,
\item[iii)] $\veps_j\alpha_{j,g}=\veps'_j(\alpha_{j,g}s_0)=-\alpha_{j,g+a_j}$,
$\veps_j(\alpha_{j,g}s_0)=\veps'_j\alpha_{j,g}=\alpha_{j,g+a_j}s_0$,
\item[iv)] $\veps_j\alpha'_{j,h}=\veps'_j(\alpha'_{j,h}s_0)=-\bi\sigma_j(h)\alpha'_{j,h+a_j}$, $\veps_j(\alpha'_{j,h}s_0)=\veps'_j\alpha'_{j,h}=\bi\sigma_j(h)\alpha'_{j,h+a_j}s_0$,
\item[v)] $\veps_j\alpha_{k,g}=\veps_j\alpha'_{k,h}=\veps'_j\alpha_{k,g}=\veps'_j\alpha'_{k,h}=0$ if $j\neq k$,
\item[vi)] $\alpha_{j,g}^2=(\alpha_{j,g}s_0)^2=-8\veps_j$, $\alpha_{j,g}\alpha_{j,g+a_j}=-(\alpha_{j,g}s_0)(\alpha_{j,g+a_j}s_0)=8$,
\item[vii)] $\alpha'^2_{j,h}=(\alpha'_{j,h}s_0)^2=8\veps'_j$, $\alpha'_{j,h}\alpha'_{j,h+a_j}=-(\alpha'_{j,h}s_0)(\alpha'_{j,h+a_j}s_0)=8\bi\sigma_j(h)s_0$,
\item[viii)] $\alpha_{j,g}\alpha_{j,g'}=\alpha_{j,g}(\alpha_{j,g'}s_0)=(\alpha_{j,g}s_0)(\alpha_{j,g'}s_0)=0$ if $g'\notin\{ g,g+a_j \}$,
\item[ix)] $\alpha'_{j,h}\alpha'_{j,h'}=\alpha'_{j,h}(\alpha'_{j,h'}s_0)=(\alpha'_{j,h}s_0)(\alpha'_{j,h'}s_0)=0$
if $h'\notin\{ h,h+a_j \}$,
\item[x)] $\alpha_{j,g}\alpha'_{j,h}=(\alpha_{j,g}s_0)\alpha'_{j,h}=\alpha_{j,g}(\alpha'_{j,h}s_0)=(\alpha_{j,g}s_0)(\alpha'_{j,h}s_0)=0$,
\item[xi)] $\alpha_{j,g}\alpha_{j+1,g'}=(\alpha_{j,g} s_0)(\alpha_{j+1,g'} s_0)=2\gamma(g,g')\alpha_{j+2,g+g'+a_{j+2}}$,
\item[xii)] $(\alpha_{j,g}s_0)\alpha_{j+1,g'}=\alpha_{j,g}(\alpha_{j+1,g'} s_0)=-2\gamma(g,g')\alpha_{j+2,g+g'+a_{j+2}} s_0$,
\item[xiii)] $\alpha_{j,g}\alpha'_{j+1,h}=(\alpha_{j,g}s_0)(\alpha'_{j+1,h}s_0)=2\bi\sigma_{j+1}(h)\gamma(g,h)\alpha'_{j+2,g+h+a_{j+2}}$,
\item[xiv)] $(\alpha_{j,g}s_0)\alpha'_{j+1,h}=\alpha_{j,g}(\alpha'_{j+1,h}s_0) =-2\bi\sigma_{j+1}(h)\gamma(g,h)\alpha'_{j+2,g+h+a_{j+2}}s_0$,
\item[xv)] $\alpha'_{j,h}\alpha_{j+1,g}=(\alpha'_{j,h}s_0)(\alpha_{j+1,g}s_0)=2\bi\sigma_j(h)\gamma(h,g)\alpha'_{j+2,h+g+a_{j+2}}$,
\item[xvi)] $(\alpha'_{j,h}s_0)\alpha_{j+1,g}=\alpha'_{j,h}(\alpha_{j+1,g}s_0)=-2\bi\sigma_j(h)\gamma(h,g)\alpha'_{j+2,h+g+a_{j+2}}s_0$,
\item[xvii)] $\alpha'_{j,h}\alpha'_{j+1,h'}=(\alpha'_{j,h}s_0)(\alpha'_{j+1,h'}s_0)
=-2\gamma(h+a_j,h'+a_{j+1})\alpha_{j+2,h+h'+a_{j+2}}s_0$,
\item[xviii)] $(\alpha'_{j,h}s_0)\alpha'_{j+1,h'}=\alpha'_{j,h}(\alpha'_{j+1,h'}s_0)
=2\gamma(h+a_j,h'+a_{j+1})\alpha_{j+2,h+h'+a_{j+2}}$.
\end{itemize}
\end{proposition}

\begin{proof}
For the first assertion, observe that $x$ and $y$ are symmetric with respect to the involution, while $xy$ is symmetric if $\deg(x)+\deg(y)\neq g_0$ and skew otherwise.

Equations from i) to x) are easily checked.  For iv) and vii), we use the property $\sigma_j(h+a_j)=-\sigma_j(h)$, which is a consequence of  $\sigma(a_j,a_j)=-1$ and the multiplicativity of $\sigma$ in the second variable.

The first equation in all cases from xi) to xviii) is easy to check, too. Also note that $(\alpha_{j,g}s_0)\alpha_{j+1,g'}=-(\alpha_{j,g}\alpha_{j+1,g'})s_0$, so case xii) is a consequence of xi). Similarly, cases xiv), xvi) and xviii) are consequences of xiii), xv) and xvii), respectively. It remains to check the second equation for the cases xi), xiii), xv) and xvii).

In xi), equation $\alpha_{j,g}\alpha_{j+1,g'}=2\gamma(g,g')\alpha_{j+2,g+g'+a_{j+2}}$ can be established using property $(P_{11})$. Indeed, $\alpha_{j,g}\alpha_{j+1,g'}=\eta(2\iota_{j+2}(\bar x_{g+a_j} \bar x_{g'+a_{j+1}})) + \eta'(2\iota_{j+2}(\bar x_g \bar x_{g'}))=2\gamma(g,g')\alpha_{j+2,g+g'+a_{j+2}}$, because $\bar x_{g+a_j} \bar x_{g'+a_{j+1}}=\gamma(g+a_j,g'+a_{j+1})x_{g+g'+a_{j+2}}=\gamma(g,g')x_{g+g'+a_{j+2}}$ and $\bar x_g \bar x_{g'}=\gamma(g,g')x_{g+g'}$.

In xvii), we use property $(P_{22})$ to obtain $\alpha'_{j,h} \alpha'_{j+1,h'} = \eta(2 \iota_{j+2} (\bar x_{h+a_j} \bar x_{h'+a_{j+1}})) + \eta'(2\iota_{j+2}(- \sigma_j(h)\sigma_{j+1}(h') \bar x_h \bar x_{h'}))
= -2\gamma(h+a_j,h'+a_{j+1}) [\eta(\iota_{j+2}(-x_{h+h'+a_{j+2}})) + \eta'(\iota_{j+2}(x_{h+h'}))]
=-2\gamma(h+a_j,h'+a_{j+1}) \alpha_{j+2,h+h'+a_{j+2}}s_0 $.

Finally, by property $(P_{12})$, respectively ($P_{21}$), and using the fact $\sigma_j(h+a_j)=-\sigma_j(h)$, we can deduce with the same arguments as above that $\alpha_{j,g}\alpha'_{j+1,h}=2\bi\sigma_{j+1}(h)\gamma(g,h)\alpha'_{j+2,g+h+a_{j+2}}$ and
$\alpha'_{j,h}\alpha_{j+1,g}=2\bi\sigma_j(h)\gamma(h,g)\alpha'_{j+2,h+g+a_{j+2}}$. This completes cases xiii) and xv).
\end{proof}

Clearly, all products in Proposition \ref{constants} are either zero or have the correct degree to make \eqref{df:Diego_grad} a $\ZZ_4^3$-grading of the algebra $\cA$. Moreover, $\ZZ_4^3$ is the universal grading group.

\begin{corollary}
The grading given by \eqref{df:Diego_grad} restricts to a $\ZZ_4^2$-grading on the subalgebra spanned by $\{1,\veps_j,\alpha_{j,g}\;|\;g\in\supp\cQ \}$, which is isomorphic to $\HK\cong\Jord{M_4(\FF)}$.
\end{corollary}

\section{A recognition theorem}\label{s:recognition}

The goal of this section is to prove the following result:

\begin{theorem}\label{th:recognitionZ4}
Let $\cA$ be the Brown algebra over an algebraically closed field $\FF$, $\chr\FF\ne 2,3$. Then, up to equivalence, there is a unique $\ZZ_4^3$-grading on $\cA$ such that all nonzero homogeneous components have dimension $1$.
\end{theorem}

To this end, we will need some general results about gradings on $\cA$ and the action of the group $\Aut(\cA,\invol)$, which contains an algebraic group of type $E_6$ as a subgroup of index $2$ (see \cite{G01}). The arguments in \cite{G01} also give that $\Der(\cA,\invol)$ is the simple Lie algebra of type $E_6$ (see also \cite{A79}).
We will use the model of $\cA$ described in Subsection \ref{ss:Brown_via_matr}. We assume that $\FF$ is algebraically closed and $\chr\FF\ne 2$, although some of the results do not require algebraic closure.

\subsection{Group gradings on $\cA$}

Recall from \eqref{df:trace_form_A} the trace form on $\cA$ and the bilinear form $\langle a,b \rangle=\tr(a\bar b)$.

\begin{lemma}\label{lm:trace_form_A}
The trace form on $\cA$ has the following properties:
\begin{itemize}
\item[i)] If $a^2=0$ and $\bar a=a$, then $\tr(a)=0$.
\item[ii)] $\langle a,b \rangle$ is a nondegenerate symmetric bilinear form.
\item[iii)] $\langle a,b \rangle$ is an invariant form: $\langle \bar a,\bar b\rangle=\langle a,b\rangle$ and $\langle ca,b\rangle=\langle a,\bar cb\rangle$.
\item[iv)] For any group grading $\cA=\bigoplus_{g\in G}\cA_g$, $gh\neq e$ implies $\langle \cA_g,\cA_h \rangle=0$.
\end{itemize}
\end{lemma}

\begin{proof}
i) Since $\bar a=a$, we have $a=\matr{\alpha & x \\ x' & \alpha}$ and $\tr(a)=2\alpha$.  Moreover,
\begin{equation}\label{eq:square}
0=a^2= \begin{pmatrix} \alpha^2+T(x,x') & 2\alpha x +x'\times x' \\ 2\alpha x'+x\times x & \alpha^2+T(x,x') \end{pmatrix},
\end{equation}
so $\alpha^2+T(x,x')=0$, $\alpha x = -x'^\#$ and $\alpha x'=-x^\#$. In case $x=0$ or $x'=0$, we have $0=\tr(a^2)=2\alpha^2$, so $\alpha=0$ and hence $\tr(a)=0$. Now assume that $x\neq 0\ne x'$ but $\alpha\ne 0$. Since $(x^\#)^\#=N(x)x$ for any $x\in\alb$ (see e.g. \cite[Eq.(4)]{MC69}), we get $\alpha x=-x'^\#=-(-\alpha^{-1}x^\#)^\#=-\alpha^{-2}N(x)x$. Thus $-\alpha^3x=N(x)x$, and similarly $-\alpha^3x'=N(x')x'$, which implies $N(x)=N(x')=-\alpha^3\neq0$. But then $T(x,x')=T(x(-\alpha^{-1})x^\#)=-3\alpha^{-1}N(x)=3\alpha^2$ and $\alpha^2+T(x,x')=4\alpha^2\neq0$, which contradicts the equation $\alpha^2+T(x,x')=0$. Therefore, $\alpha=0$ and $\tr(a)=0$.

ii) Since $\tr$ is invariant under the involution, $\langle a,b\rangle=\tr(a\bar b)=\tr(\overline{a\bar b})=\tr(b\bar a)=\langle b,a\rangle$, so  $\langle\cdot,\cdot\rangle$ is symmetric. The nondegeneracy of the bilinear form $\langle\cdot,\cdot\rangle$ is a consequence of the nondegeneracy of the trace form $T$ of $\alb$.

iii) It is easy to see that $\tr(ab)=\tr(ba)$. Hence $\langle \bar a,\bar b\rangle=\tr(\bar ab)=\tr(b\bar a)=\tr(\overline{b \bar a})=\tr(a \bar b)=\langle a,b\rangle$. Using the fact that $T(x\times y,z)=N(x,y,z)$ is symmetric in the three variables, is is straightforward to check that $\langle ca,b\rangle=\langle a,\bar cb\rangle$.

iv) Observe that the restriction of $\tr$ to the subspace $\cA_0\bydef\FF s_0\oplus\ker(\id+L_{s_0}R_{s_0})$ is zero, and $\cA=\FF 1\oplus\cA_0$, so $\cA_0$ equals the kernel of $\tr$. Now, $\FF s_0$ is a graded subspace and $s_0^2=1$, hence $s_0$ is a homogeneous element and its degree has order at most $2$. It follows that $\cA_0$ is a graded subspace. Therefore, $\cA_g\subset\cA_0$ for any $g\ne e$. The result follows.
\end{proof}

\begin{lemma}\label{lm:subalgebra}
For any $G$-grading on $\cA$ and a subgroup $H\subset G$ such that $\deg(s_0)\notin H$, $\cB=\bigoplus_{h\in H}\cA_h$ is a semisimple Jordan algebra of degree $\le 4$.
\end{lemma}

\begin{proof}
Since $\deg(s_0)\notin H$, the involution is trivial on $\cB$, so $\cB$ is a Jordan algebra. By Lemma \ref{lm:trace_form_A}(ii), the symmetric form $\langle\cdot,\cdot\rangle$ is nondegenerate on $\cA$. By (iv), the subspaces $\cA_g$ and $\cA_{g^{-1}}$ are paired by $\langle\cdot,\cdot\rangle$ for any $g\in G$. It follows that the restriction of $\langle\cdot,\cdot\rangle$ to $\cB$ is nondegenerate. Moreover, (iii) implies that this restriction is associative in the sense $\langle ab,c\rangle=\langle a,bc\rangle$ for all $a,b,c\in\cB$.

Suppose $\cI$ is an ideal of $\cB$ satisfying $\cI^2=0$. For any $a\in\cI$ and $b\in\cB$, we have $ab\in\cI$ and hence $(ab)^2=0$. By Lemma \ref{lm:trace_form_A}(i), this implies $\tr(ab)=0$. We have shown that $\langle\cI,\cB\rangle=0$, so $\cI=0$. By Dieudonn\'e's Lemma, we conclude that $\cB$ is a direct sum of simple ideals.

The conjugate norm of a structurable algebra was defined in \cite{AF92} as the exact denominator of the (conjugate) inversion map (i.e., the denominator of minimal degree), and it coincides with the generic norm in the case of a Jordan algebra. If $N_\cB$ is the generic norm of $\cB$, then it is the denominator of minimal degree for the inversion map, and therefore it divides any other denominator for the inversion map. Since the conjugate norm of $\cA$ has degree $4$, we conclude that the degree of $N_\cB$ is at most $4$.
\end{proof}

\begin{lemma}\label{lm:subalgebra_bis}
For any $G$-grading on $\cA$ and a subgroup $H\subset G$ such that $\deg(s_0)\in H$, $\cB=\bigoplus_{h\in H}\cA_h$ is a simple structurable algebra of skew-dimension $1$.
\end{lemma}

\begin{proof}
If $\cI$ is an ideal of $\cB$ as an algebra with involution and $\cI^2=0$, then $s_0\notin \cI$, so $\cI$ is a Jordan algebra, and, as in the proof of Lemma \ref{lm:subalgebra}, we obtain $\cI=0$. On the other hand, if $\cI\ne 0$ is an ideal of $\cB$ as an algebra (disregarding involution),  $\cI^2=0$, and $\cI$ is of minimal dimension with this property, then either $\cI=\bar\cI$ or $\cI\cap\bar\cI=0$. In the first case, $\cI$ is an ideal of $\cB$ as an algebra with involution, so we get a contradiction. In the second case, $\cI\oplus\bar\cI$ is an ideal of $\cB$ as an algebra with involution and $(\cI\oplus\bar\cI)^2=0$, again a contradiction. The bilinear form $(a\vert b)\bydef\langle a,\bar b\rangle$ is symmetric, nondegenerate and associative on $\cA$, and hence on $\cB$. Therefore, Dieudonn\'e's Lemma applies and tells us that $\cB$ is a direct sum of simple ideals (as an algebra). The involution permutes these ideals so, adding each of them with its image under the involution, we write $\cB$ as a direct sum of ideals, each of which is simple as an algebra with involution. Since $\dim\sks(\cB,\invol)=1$, there is only one such ideal where the involution is not trivial, and it contains $s_0$. Since $s_0^2=1$, this ideal is the whole $\cB$.
\end{proof}

\subsection{Norm similarities of the Albert algebra}

A linear bijection $f\colon\alb\to\alb$ is called a {\em norm similarity with multiplier $\lambda$} if $N(f(x))=\lambda N(x)$ for all $x\in\alb$. Norm similarities with multiplier $1$ are called {\em (norm) isometries}. We will denote the group of norm similarities by  $M(\alb)$ and the group of isometries by $M_1(\alb)$.

For $f\in\text{End}(\alb)$, denote by $f^*$ the adjoint with respect to the trace form $T$ of $\alb$, i.e., $T(f(x),y)=T(x,f^*(y))$ for all $x,y\in\alb$. Following the notation of \cite{G01}, for any element $\varphi\in M(\alb)$, we denote the element $(\varphi^*)^{-1}=(\varphi^{-1})^*$ by $\varphi^\dagger$, so we have $T(\varphi(x),\varphi^\dagger(y))=T(x,y)$ for all $x,y\in\alb$. If the multiplier of $\varphi$ is $\lambda$, then $\varphi^\dagger$ is a norm similarity with multiplier $\lambda^{-1}$, and also
\begin{equation}\label{eq:dagger}
\varphi(x)\times\varphi(y)=\lambda\varphi^\dagger(x\times y) \quad \text{and} \quad \varphi^\dagger(x)\times\varphi^\dagger(y)=\lambda^{-1}\varphi(x\times y)
\end{equation}
for all $x,y\in\alb$ (see \cite[Lemma~1.7]{G01}). The $U$-operator $U_x(y)\bydef\{x,y,x\}=2x(xy)-x^2y$ can also be written as $U_x(y)=T(x,y)x-x^\#\times y$ (see \cite[Theorem~1]{MC70}; cf. \cite[Theorem~1]{MC69}). Therefore, $U_{\varphi(x)}\varphi^\dagger(y)=U_x(y)$ for any $\varphi\in M(\alb)$ and $x,y\in\alb$. It follows that the automorphisms of the Albert algebra are precisely the elements $\varphi\in M_1(\alb)$ such that $\varphi^\dagger=\varphi$. Moreover, any $\varphi\in M_1(\alb)$ defines an automorphism of the Brown algebra $\cA$ given by
\begin{equation*}
\begin{pmatrix} \alpha & x \\ x' & \beta \end{pmatrix} \mapsto
\begin{pmatrix} \alpha & \varphi(x) \\ \varphi^\dagger(x') & \beta \end{pmatrix}.
\end{equation*}
Thus we can identify $M_1(\alb)$ with a subgroup of $\Aut(\cA,\invol)$. In fact, this subgroup is precisely the stabilizer of the element $s_0$.

For $\lambda_1,\lambda_2,\lambda_3\in \FF^\times$ and $\mu_i=\lambda^{-1}_i\lambda_{i+1}\lambda_{i+2}$, we can define a norm similarity $c_{\lambda_1,\lambda_2,\lambda_3}$, with multiplier $\lambda_1\lambda_2\lambda_3$, given by $\iota_i(x)\mapsto\iota_i(\lambda_i x)$, $E_i\mapsto\mu_i E_i$. Note that $c^\dagger_{\lambda_1,\lambda_2,\lambda_3}$ is given by $\iota_i(x)\mapsto\iota_i(\lambda_i^{-1} x)$, $E_i\mapsto\mu_i^{-1} E_i$. (These norm similarities appear e.g. in \cite[Eq.~(1.6)]{G01}.)
For $\lambda\in\FF^\times$, denote $c_\lambda:=c_{\lambda,\lambda,\lambda}$.

\begin{df}\label{deforbits}
We define the {\it rank} of $x\in\alb$ by $\rank(x)\bydef\dim(\im U_x)$ and denote $\cO_n\bydef\{x\in \alb\;|\;\rank(x)=n\}$.
\end{df}

Since $U_{\varphi(x)}\varphi^\dagger(y)=U_x(y)$ for $\varphi\in M(\alb)$, the rank is invariant under the action of $M(\alb)$.
Actually, we will show now that the rank of an element $x\in\alb$ determines its $M(\alb)$-orbit. Denote by $\mu_x(X)$ the minimal polynomial of $x$; it is a divisor of the generic minimal polynomial $m_x(X)=X^3-T(x)X^2+S(x)X-N(x)$.

\begin{lemma}[{\cite[Exercise 6, p.393]{J68}}]\label{lm:transitivity_M1}
$M_1(\alb)$ is transitive on the set of elements of generic norm 1.\qed
\end{lemma}

\begin{proposition}\label{prop:orbitsA}
The orbits for the action of $M(\alb)$ on $\alb$ are exactly $\cO_0=\{0\}$, $\cO_1$, $\cO_{10}$ and $\cO_{27}$. The orbit $\cO_{27}$ consists of all nonisotropic elements: $x\in\alb$ with $N(x)\ne 0$. The orbit $\cO_1$ consists of all $0\neq x\in\alb$ satisfying $N(x)=0$, $S(x)=0$ and $\deg\mu_x=2$; in this case $\mu_x(X)=X^2-T(x)X$.
\end{proposition}

\begin{proof}
It is clear that $\cO_0$ is an orbit. $\cO_{27}$ consists of all elements that are invertible in the Jordan sense, which are precisely the nonisotropic elements. Hence $\cO_{27}$ is an orbit by Lemma \ref{lm:transitivity_M1}. It remains to consider the orbits of the isotropic nonzero elements.

By \cite[Chapter~IX, Theorem~10]{J68}, two elements of $\alb$ are in the same $\Aut(\alb)$-orbit if and only if they have the same minimal polynomial and the same generic minimal polynomial. Take $0\neq x\in\alb$ with $N(x)=0$ and consider two possible cases: $\mu_x=m_x$ (i.e., $\deg\mu_x=3$) and $\mu_x\neq m_x$ (i.e., $\deg\mu_x(X)=2$). Since $N(x)=0$, we have $X|m_x(X)$, but $\mu_x(X)$ and $m_x(X)$ have the same irreducible factors by \cite[Chapter~VI, Theorem~1]{J68}, so $X|\mu_x(X)$, too. Thus, if $\deg\mu_x=2$ then $\mu_x(X)=X^2+\lambda X$, and, by the same result, either $m_x(X)=X^2(X+\lambda)$ or $m_x(X)=X(X+\lambda)^2$.

1) Case $\deg\mu_x=3$. Then $\mu_x(X)=X^3+\lambda X^2+\mu X$.
\begin{itemize}
\item If $\lambda=0\neq\mu$, then $x_0=\sqrt{\mu}\bi(E_2-E_3)$ is a representative of the $\Aut(\alb)$-orbit of $x$. By applying an appropriate norm similarity $c_{\alpha,\beta,\gamma}$ to $x_0$, we see that it is in the $M(\alb)$-orbit of $\widetilde{E}:=E_2+E_3$.

\item If $\lambda=\mu=0$, then $x_0=\iota_2(1)+\iota_3(\bi)$ is a representative of the $\Aut(\alb)$-orbit of $x$. But using an appropriate norm similarity  $c_{\alpha,\beta,\gamma}$, we see that $x_0$ is in the $M(\alb)$-orbit of $\iota_2(1)+\iota_3(1)$ ($\lambda=0$ and  $\mu=-8$), which is the orbit of $\widetilde{E}$ by the previous case.

\item If $\lambda\neq0\neq\mu$, we can find $\alpha,\beta\in \FF^\times$ such that $\alpha+\beta=-\lambda$ and $\alpha\beta=\mu$.
If $\alpha\neq\beta$, then $\alpha E_2+\beta E_3$ is a representantive of the  $\Aut(\alb)$-orbit of $x$, and it is the $M(\alb)$-orbit of $\widetilde{E}$. If $\alpha=\beta$, then applying $c_{\alpha^{-1}}$ we may assume $\lambda=-2$ and $\mu=1$, so $x_0=2E_1+\frac{\bi}{2}\iota_2(1)$ is a representative of the $\Aut(\alb)$-orbit of $x$. Applying now $c_{2\sqrt{2},\sqrt{2},2}$, we move $x_0$ to the element $2E_1+\frac{\bi}{\sqrt{2}}\iota_2(1)$, which is in the $\Aut(\alb)$-orbit of $(1+\bi)E_2+(1-\bi)E_3$ ($\lambda=-2$ and $\mu=2$), and therefore in the $M(\alb)$-orbit of $\widetilde{E}$.

\item If $\lambda\ne 0=\mu$, then $x_0=-\lambda(E_1+\iota_2(1)+\iota_3(\bi))$ is a representative of the $\Aut(\alb)$-orbit of $x$. Applying $c_{-\lambda^{-1},-\lambda^{-1},\lambda^{-1}\bi}$, we move $x_0$ to the element $-\bi E_1+\iota_2(1)+\iota_3(1)$, which has $\lambda=\bi$ and $\mu=-8$, and hence belongs to the $M(\alb)$-orbit of $\widetilde{E}$ by the previous case.
\end{itemize}

2) Case $\deg\mu_x(X)=2$. Then $\mu_x(X)=X^2+\lambda X$.
\begin{itemize}
\item If $\lambda\neq0$, then in the case $m_x(X)=X^2(X+\lambda)$, the element $-\lambda E_1$ is a representative of the $\Aut(\alb)$-orbit of $x$, whereas in the case $m_x(X)=X(X+\lambda)^2$, the element $-\lambda(E_2+E_3)$ is a representative of the  $\Aut(\alb)$-orbit of $x$. Clearly, these elements are in the $M(\alb)$-orbits of $E_1$ and $\widetilde{E}$, respectively.
\item If $\lambda=0$ then $m_x(X)=X^3$ and $x_0=2E_1-2E_2+\iota_3(\bi)$ is a representative of the  $\Aut(\alb)$-orbit of $x$. But using $c_{\bi,1,\bi}$, we see that $x_0$ is in the $M(\alb)$-orbit of the element $2E_1+2E_2-\iota_3(1)$, whose minimal polynomial is  $X^2-4X$, so it falls under the previous case.
\end{itemize}

We conclude that the only nontrivial isotropic orbits are the ones of $\widetilde{E}$ and $E_1$. Since $\rank(E_1)=1$ and $\rank(\widetilde{E})=10$, these orbits are different. The characterization of $\cO_1$ follows from the cases considered above.
\end{proof}

\begin{remark} \label{remarkrank1}
For any nonzero isotropic element $x\in\alb$, we have $x\in\cO_1$ if and only if $x^\#=0$ (and therefore, $x\in\cO_{10}$ if and only if $x^\#\neq0$). Indeed, if $x\in\cO_1$, then $x^\#=x^2-T(x)x=0$ by Proposition~\ref{prop:orbitsA}. Conversely, if $x^\#=0$, then $\deg\mu_x=2$ and, since $S(x)=T(x^\#)$, we also have $S(x)=0$, hence $x\in\cO_1$ by Proposition~\ref{prop:orbitsA}. Note that in \cite{J68}, the elements of rank $1$ are defined as the elements $x\ne 0$ such that $x^\#=0$ (see p.364), which is equivalent to our definition.
\end{remark}

\begin{corollary}\label{cor:orbitsA}
The orbits for the action of $M_1(\alb)$ on $\alb$ are $\cO_0$, $\cO_1$, $\cO_{10}$ and $\cO_{27}(\lambda)\bydef\{x\in\cO_{27}\;|\; N(x)=\lambda\}$, $\lambda\in\FF^\times$.
\end{corollary}

\begin{proof}
Note that the elements $E_1$ and $E_2+E_3$ can be scaled by any $\lambda\in\FF^\times$ using some norm similarity $c_{\alpha,\beta,\gamma}$ with $\alpha\beta\gamma=1$. Therefore, $\cO_1$ and $\cO_{10}$ are orbits for $M_1(\alb)$, too. The fact that $\cO_{27}(\lambda)$ is an orbit for $M_1(\alb)$ follows from Lemma \ref{lm:transitivity_M1}.
\end{proof}

\begin{lemma}\label{lemmarank1}\rm The rank function on $\alb$ has the following properties:
\begin{itemize}
\item[i)] If $x,y\in\alb$ have rank $1$, then $N(x+y)=0$.
\item[ii)] If $x_1,x_2,x_3\in\alb$ have rank $1$ and $N(x_1+x_2+x_3)\ne 0$, then $x_i+x_j$ has rank $10$ for each $i\neq j$.
\item[iii)] If $x_1,x_2,x_3\in\alb$ have rank $1$ and $N(x_1+x_2+x_3)=1$, then there is an isometry sending $x_i$ to $E_i$ for all $i$.
\item[iv)] If $\rank(x)=1$, then $\rank(x^\#)=0$. If $\rank(x)=10$, then $\rank(x^\#)=1$. If $\rank(x)=27$, then $\rank(x^\#)=27$. In general, $\rank(x^\#)\leq\rank(x)$.
\end{itemize}
\end{lemma}

\begin{proof}
i) Assume, to the contrary, that $N(x+y)\ne 0$. By Lemma \ref{lm:transitivity_M1}, applying a norm similarity, we may assume $x+y=1$. We know by Proposition \ref{prop:orbitsA} that $x^2=T(x) x$. If it were $T(x)=0$, applying an automorphism of $\alb$ we would have $x=\iota_1(a)$ with $n(a)=0$, and therefore $N(y)=N(1-\iota_1(a))\neq0$, which contradicts $\rank(y)=1$. Thus $\lambda:=T(x)\ne 0$. Hence, applying an automorphism of $\alb$, we may assume $x=\lambda E_1$, and we still have $x+y=1$. If $\lambda=1$, then $S(y)=1\ne 0$; otherwise $N(y)=1-\lambda\ne 0$. By Proposition~\ref{prop:orbitsA}, in both cases we get a contradiction: $y\notin\cO_1$.

ii) Take $k$ such that $\{i,j,k\}=\{1,2,3\}$. By i), $\rank(x_i+x_j)\neq27$. We cannot have $\rank(x_i+x_j)=0$, because this would imply $x_i+x_j=0$ and $\rank(x_k)=27$. We cannot have $\rank(x_i+x_j)=1$, because this would imply $N(x_i+x_j+x_k)=0$ by i). Therefore, $\rank(x_i+x_j)=10$.

iii) Applying an isometry, we may assume $x_1+x_2+x_3=1$. By ii), we have $\rank(x_i+x_{i+1})=10$. By Proposition~\ref{prop:orbitsA}, we know that $x_i^2=T(x_i)x_i$. If it were $T(x_1)=0$, applying an automorphism of $\alb$ we would have $x_1=\iota_1(a)$ with $n(a)=0$, and therefore $N(x_2+x_3)=N(1-\iota_1(a))\neq0$, which contradicts $\rank(x_2+x_3)=10$. Hence, $T(x_i)\ne 0$ for $i=1,2,3$. Applying an automorphism of $\alb$, we obtain $x_1=\lambda E_1$ where $\lambda=T(x_1)$, and still $x_1+x_2+x_3=1$. If $\lambda\ne 1$, then $N(x_2+x_3)=1-\lambda\ne 0$, which contradicts i). Therefore, $T(x_1)=1$, and similarly $T(x_2)=T(x_3)=1$. We have shown that the $x_i$ are idempotents. Moreover, since $1-x_i=x_{i+1}+x_{i+2}$ is an idempotent, we also have $x_{i+1}x_{i+2}=0$, so the idempotents $x_i$ are orthogonal with $\sum x_i=1$. Now by \cite[Chapter~IX, Theorem~10]{J68}, there exists an automorphism of $\alb$ sending $x_i$ to $E_i$ for $i=1,2,3$.

iv) If $\rank(x)=1$, we already know that $x^\#=0$. It follows from \eqref{eq:dagger} that $\varphi(x)^\#=\varphi^\dagger(x^\#)$ for any isometry $\varphi$. If $\rank(x)=10$, then by Corollary \ref{cor:orbitsA} there is $\varphi\in M_1(\alb)$ such that $\varphi(x)=E_2+E_3$, hence $\varphi^\dagger(x^\#)=\varphi(x)^\#=(E_2+E_3)^\#=E_1$, and so $\rank(x^\#)=1$. If $\rank(x)=27$, then $N(x)\neq0$. Since $N(x^\#)=N(x)^2$ (see \cite{MC69}), we obtain $N(x^\#)\ne 0$ and $\rank(x^\#)=27$.
\end{proof}

\subsection{Proof of Theorem \ref{th:recognitionZ4}}

Suppose $\Gamma:\;\cA=\bigoplus_{g\in\ZZ_4^3}\cA_g$ is a grading such that $\dim\cA_g\le 1$ for all $g\in\ZZ_4^3$. Set $g_0=\deg(s_0)$, so $g_0$ is an element of order $2$.

Denote $W=\eta(\alb)\oplus\eta'(\alb)$. Since $W=\ker(\id+L_{s_0}R_{s_0})$, it is a graded subspace. Hence, for any $g\ne 0,g_0$, we have $\cA_g\subset W$. Also, for any $g\ne g_0$, the component $\cA_g$ consists of symmetric elements.

Let $S_{g_0}=\{g\in\ZZ_4^3\;|\;2g\ne g_0\}$. We claim that $\supp\Gamma=S_{g_0}$. Note that $|S_{g_0}|=56=\dim\cA$, so it is sufficient to prove that $2g=g_0$ implies $\cA_g=0$. Assume, to the contrary, that $0\ne a\in\cA_g$. Then $b=as_0$ is a nonzero element in  $\cA_{-g}$. By Lemma \ref{lm:trace_form_A}, the components $\cA_g$ and $\cA_{-g}$ are in duality with respect to the form $\langle\cdot,\cdot\rangle$, hence $\langle a,b\rangle\ne 0$. But $a=\eta(x)+\eta'(x')$ for some $x,x'\in\alb$, so $b=-\eta(x)+\eta'(x')$, which implies $\langle a,b\rangle=\tr(ab)=T(x,x')-T(x,x')=0$, a contradiction.

Suppose $H$ is a subgroup of $\ZZ_4^3$ isomorphic to $\ZZ_4^2$ and not containing $g_0$. Consider $\cB=\bigoplus_{h\in H}\cA_h$ and $\cD=\cB\oplus s_0\cB$.  Lemma \ref{lm:subalgebra_bis} shows that $\cD$ is a simple structurable algebra of skew-dimension $1$ and dimension $32$. Hence, by \cite[Example 1.9]{Allison90}, $\cD$ is the structurable matrix algebra corresponding to a triple $(T,N,N)$ where either (a) $N$ and $T$ are the generic norm and trace form of a degree $3$ semisimple Jordan algebra $J$, or (b) $N=0$ and $T$ is the generic trace form of the Jordan algebra $J=\cJ(V)$ of a vector space $V$ with a nondegenerate symmetric bilinear form. In case (a), we have by dimension count that either $J=\cH_3(\cQ)$ or $J=\FF\times \cJ(V)$, where $\dim V=13$.
In case (a), as in Subsection \ref{ss:Brown_via_matr}, $\FF\times J$ is a Jordan subalgebra of $\cD$. If $J=\FF\times \cJ(V)$, then  $\cL\bydef\lspan{D_{x,y}\;|\;x,y\in V}$ (the operators $D_{x,y}$ are defined by \eqref{eq:Dxy} in the next section) is a subalgebra of $\Der(\cA,\invol)$ isomorphic to the orthogonal Lie algebra $\frso(V)$. Indeed, the image of $\cL$ in $\End(V)$ is $\frso(V)$, and $\dim\cL\leq \wedge^2V=\dim\frso(V)$. But $\dim\frso(V)=78=\dim\Der(\cA,\invol)$ and $\Der(\cA,\invol)$ is simple of type $E_6$, so we obtain a contradiction.
In case (b), $\cD$ contains the Jordan algebra of a vector space of dimension $15$ (the Jordan algebra $J$ with its generic trace form), hence $\Der(\cA,\invol)$ contains a Lie subalgebra isomorphic to $\frso_{15}(\FF)$, which has dimension larger than $78$, so we again obtain a contradiction.
Therefore, the only possibility is $J=\cH_3(\cQ)$. Then, with the same arguments as for $(\cA,\invol)$, it can be shown that $\Der(\cD,\invol)$ is a simple Lie algebra of type $A_5$, so it has dimension $35$.

By Lemma \ref{lm:subalgebra}, $\cB$ is a semisimple Jordan algebra of degree $\le 4$. Since $\dim\cB=16$, we have the following possibilities: (i) $\cJ(V)$ with $\dim V=15$, (ii) $\FF\times\cJ(V)$ with $\dim V=14$, (iii) $\FF\times\FF\times\cJ(V)$ with $\dim V=13$, (iv) $\cJ(V_1)\times\cJ(V_2)$ with $\dim V_1+\dim V_2=14$ and $\dim V_i\ge 2$, (v) $\FF\times\sym_3(\cQ)$ and (vi) $\Jord{M_4(\FF)}$, where, as before, $\cJ(V)$ denotes the Jordan algebra of a vector space $V$ with a nondegenerate symmetric bilinear form. Cases (ii), (iii) and (v) are impossible, because these algebras do not admit a $\ZZ_4^2$-grading with $1$-dimensional components. Indeed, since $\chr\FF\ne 2$, such a grading would be the eigenspace decomposition with respect to a family of automorphisms, but in each case there is a subalgebra of dimension $2$ whose elements are fixed by all automorphisms. The same argument applies in case (iv) unless $\dim V_1=\dim V_2=7$. On the other hand, cases (i) and (iv) give, as in the previous paragraph, subalgebras of $\Der(\cD,\invol)$ isomorphic to  $\frso(V)$ or $\frso(V_1)\times\frso(V_2)$ of dimension larger than $35$, so these cases are impossible too.
We are left with case (vi), i.e., $\cB\cong\Jord{M_4(\FF)}$. Then, up to equivalence, there is only one $\ZZ_4^2$-grading with $1$-dimensional components, namely, the Pauli grading on the associative algebra $M_4(\FF)$. (For the classification of gradings on simple special Jordan algebras, we refer the reader to \cite[\S 5.6]{EKmon}.)

As a consequence of the above analysis, if $X\ne 0$ is a homogeneous element of $\cA$ whose degree has order $4$ then we have $0\neq X^4\in\FF 1$. Indeed, the degree of $X$ is contained in a subgroup $H$ as above, so $X$ is an invertible matrix in $\cB\cong \Jord{M_4(\FF)}$. Moreover, we can fix homogeneous elements $X_1$, $X_2$ and $X_3$ of $\cB$ such that $X_i^2=1$ and $X_{i}X_{i+1}=X_{i+2}$. We will now show that $\Gamma$ is equivalent to the grading defined by \eqref{df:Diego_grad} in Section \ref{s:MatrixStructAlgebra}. Denote $a_i=\deg(X_i)$, then the subgroup $\langle a_1,a_2\rangle$ is isomorphic to $\ZZ_2^2$ and does not contain $g_0$.

We can write $X_i=\eta(x_i)+\eta'(x'_i)$ with $x,x'\in\alb$. Since $X_i^2=1$, we get $x_i^\#=0=x'^\#_i$ and thus $x_i$ and $x'_i$ have rank 1 (see Remark~\ref{remarkrank1}). Set $Z=X_1+X_2+X_3$ and write $Z=\eta(z)+\eta(z')$ with $z,z'\in\alb$. Then $Z^2=2Z+3$, which implies $z^\#=z'$ and $z'^\#=z$. But, by Lemma \ref{lemmarank1}(iv), $\rank(z^\#)\leq\rank(z)$ and $\rank(z'^\#)\leq\rank(z')$, so we get $\rank(z^\#)=\rank(z)=\rank(z')=\rank(z'^\#)$. Since $Z\ne 0$, we have $z\ne 0$ or $z'\ne 0$, and hence by Lemma \ref{lemmarank1}(iv), we obtain $\rank(z)=27=\rank(z')$.
Then, by Lemma \ref{lemmarank1}(iii), there is an isometry of $\alb$ sending $x_i$ to $\lambda E_i$ ($i=1,2,3$), where $\lambda$ is any element of $\FF$ satisfying $\lambda^3 = N(z)$. Since isometries of $\alb$ extend to automorphisms of $(\cA,\invol)$, we may assume that $x_i=\lambda E_i$. Then $X_i X_{i+1}=X_{i+2}$ implies $x'_i=\lambda^2 E_i$ and hence $\lambda^3=1$. Therefore, $N(z)=1$ and we may take $\lambda=1$, so $x_i=E_i=x'_i$, i.e.,  $X_i=\veps_i\bydef\eta(E_i)+\eta'(E_i)$. Thus, $\veps_i$ and $\veps'_i\bydef\veps_i s_0$ are homogeneous elements; their degrees are precisely the order $2$ elements of $\ZZ_4^3$ different from $g_0$.

Since the subspaces $\ker(L_{\veps_i})=\eta(\iota_{i+1}(\cC)\oplus\iota_{i+2}(\cC))\oplus\eta'(\iota_{i+1}(\cC)\oplus\iota_{i+2}(\cC))$ are graded, so are $\eta(\iota_i(\cC))\oplus\eta'(\iota_i(\cC))$, $i=1,2,3$. For any homogeneous element $X=\eta(\iota_j(x))+\eta'(\iota_j(x'))$, we saw that $0\neq X^4\in\FF1$, which forces $0\neq X^2\in\FF\veps_j\cup\FF\veps'_j$, and this implies  $n(x,x')=0$ and $n(x)=\pm n(x')\ne 0$. These facts will be used several times. Also note that automorphisms of $\cC$ extend to automorphisms of $\alb$ preserving $E_i$, and therefore to automorphisms of $\cA$ preserving $\veps_i$.

Fix homogeneous elements $Y_1=\eta(\iota_1(y_1))+\eta'(\iota_1(y'_1))$ and $Y_2=\eta(\iota_2(y_2))+\eta'(\iota_2(y'_2))$
such that $Y_i^2 \in \FF\veps_i$. Without loss of generality, we may assume $n(y_1) = 1 = n(y_2)$, and therefore $n(y'_1)=1=n(y'_2)$. Also, we have $n(y_i,y'_i)=0$. By \cite[Lemma 5.25]{EKmon}, there exists an automorphism of $\alb$ that fixes $E_i$ and sends $y_1$ and $y_2$ to $1$. Thus we may assume $y_1=1=y_2$ and hence $\overline{y'_i}=-y'_i$. Then $Y_1Y_2=\eta(2\iota_3(y'_1y'_2))+\eta'(2\iota_3(1))$, so we obtain $n(1,y'_1y'_2)=0$, which implies $n(y'_1,y'_2)=0$. Thus the elements $1, y'_1, y'_2$ are orthogonal of norm $1$, and applying an automorphism of $\cC$ (extended to $\cA$) we may assume that $Y_1=\alpha_{1,0}:=\eta(\iota_1(1))+\eta'(\iota_1(x_{a_1}))$ and $Y_2=\alpha_{2,0}:=\eta(\iota_2(1))+\eta'(\iota_2(x_{a_2}))$, as in the grading \eqref{df:Diego_grad}. Consequently, the elements of the form $\alpha_{j,g}$, for $j=1,2,3$ and $g\in\langle a_1,a_2\rangle$, will be homogeneous because they can be expressed in terms of $\alpha_{1,0}$ and $\alpha_{2,0}$.

Fix a new element $Y_3=\eta(\iota_3(y_3))+\eta'(\iota_3(y'_3))$ such that $Y_3^2\in\FF\veps'_3$. As before, we have $n(y_3,y'_3)=0$, but this time $n(y_3)=-n(y'_3)$. Using again that the products of the form $Y_3\alpha_{1,g}$ and $Y_3\alpha_{2,g}$, with $g\in\langle a_1,a_2\rangle$, have orthogonal entries in $\cC$, we deduce that $y_3,y'_3\in\cQ^\perp$, where $\cQ=\lspan{1, x_{a_i}\;|\;i=1,2,3}$, and that $y'_3\in\FF y_3x_{a_3}$. Hence, scaling $Y_3$, we obtain either $Y_3=\alpha'_{3,h}$ or $Y_3=\alpha'_{3,h}s_0$ for some $h\in g_0+\langle a_1,a_2\rangle$. (Actually, applying another automorphism of $\cC$ that fixes the subalgebra $\cQ$ point-wise, we can make $h$ any element we like in the indicated coset.) Replacing $Y_3$ by $Y_3 s_0$ if necessary, we may assume $Y_3=\alpha'_{3,h}$. Since the elements $\alpha_{1,0}$, $\alpha_{2,0}$ and $\alpha'_{3,h}$ determine the $\ZZ_4^3$-grading \eqref{df:Diego_grad}, the proof is complete.

\section{Fine gradings on the exceptional simple Lie algebras $E_6$, $E_7$ and $E_8$}\label{s:E}

Gradings on the exceptional simple Lie algebras are quite often related to gradings on certain nonassociative algebras that coordinatize the Lie algebra in some way. The aim of this section is to indicate how the fine grading by $\ZZ_4^3$ on the split Brown algebra is behind all the fine gradings on the simple Lie algebras of types $E_6$, $E_7$ and $E_8$ mentioned in the introduction. Here we will assume that the ground field $\FF$ is algebraically closed and $\chr\FF\ne 2,3$.

Given a structurable algebra $(\cX,\invol)$, there are several Lie algebras attached to it. To begin with, there is the Lie algebra of derivations $\Der(\cX,\invol)$. For the Brown algebra, this coincides with the Lie algebra of inner derivations $\inder(\cX,\invol)$, which is the linear span of the operators $D_{x,y}$, for $x,y\in\cX$, where
\begin{equation}\label{eq:Dxy}
D_{x,y}(z)=\frac{1}{3}\left[[x,y]+[\bar x,\bar y],z\right]+(z,y,x)-(z,\bar x,\bar y)
\end{equation}
for $x,y,z\in \cX$. (As before, $(x,y,z)$ denotes the associator $(xy)z-x(yz)$.) If $(\cX,\invol)$ is $G$-graded, then $\Der(\cX,\invol)$ is a graded Lie subalgebra of $\End(\cX)$, so we obtain an induced $G$-grading on $\Der(\cX,\invol)$. For the Brown algebra $(\cA,\invol)$,  the Lie algebra of derivations is the simple Lie algebra of type $E_6$. The fine grading by $\ZZ_4 ^3$ on the Brown algebra induces the fine grading by $\ZZ_4^3$ on $E_6$ that appears in \cite{DV_e6}.

Another Lie subalgebra of $\End(\cX)$ is the \emph{structure Lie algebra}
\[
\str(\cX,\invol)=\Der(\cX,\invol)\oplus T_{\cX}
\]
where $T_x\bydef V_{x,1}$, $x\in\cX$. The linear span of the operators  $V_{x,y}$, $x,y\in\cX$, is contained in $\str(\cX,\invol)$ and called the \emph{inner structure Lie algebra} (as it actually equals $\inder(\cX,\invol)\oplus T_{\cX}$).  It turns out (see e.g. \cite[Corollaries 3 and 5]{A78}) that $\str(\cX,\invol)$ is graded by $\ZZ_2$, with $\str(\cX,\invol)\subo=\Der(\cX,\invol)\oplus T_\sks$ and $\str(\cX,\invol)\subuno=T_\sym$, where $\sks=\sks(\cX,\invol)$ and $\sym=\sym(\cX,\invol)$ denote, respectively, the spaces of symmetric and skew-symmetric elements for the involution. If $(\cX,\invol)$ is $G$-graded then we obtain an induced grading by $\ZZ_2\times G$ on $\str(\cX,\invol)$ and on its derived algebra. In the case of the Brown algebra $(\cA,\invol)$, the inner structure Lie algebras coincides with the structure Lie algebra and is the direct sum of a one-dimensional center and the simple Lie algebra of type $E_7$. (The arguments in \cite[Corollary 7]{A79} work here because the Killing form of $E_6$ is nondegenerate.) Therefore, the $\ZZ_4^3$-grading on $(\cA,\invol)$ induces a grading by $\ZZ_2\times\ZZ_4^3$ on the simple Lie algebra of type $E_7$.

Also, the \emph{Kantor Lie algebra} $\kan(\cX,\invol)$ (see \cite{A79}) is the Lie algebra defined on the vector space
\[
\tilde\frn\oplus \str(\cX,\invol)\oplus\frn,
\]
where $\frn=\cX\times \sks$, $\tilde\frn$ is another copy of $\frn$, $\str(\cX,\invol)$ is a subalgebra and
\[
\begin{split}
&[(f,(x,s)]=\bigl(f(x),f^\delta(s)\bigr),\\
&[f,(x,s)\tilde{\ }]=\bigl(f^\varepsilon(x),f^{\varepsilon\delta}(s)\bigr)\tilde{\ },\\
&[(x,r),(y,s)]=(0,x\bar y-y\bar x),\\
&[(x,r)\tilde{\ },(y,s)\tilde{\ }]=(0,x\bar y-y\bar x)\tilde{\ },\\
&[(x,r),(y,s)\tilde{\ }]=-(sx,0)\tilde{\ }+V_{x,y}+L_rL_s+(ry,0),
\end{split}
\]
for any $x,y\in \cX$, $r,s\in\sks$, and $f\in\str(\cX,\invol)$, where $f^\varepsilon\bydef f-T_{f(1)+\overline{f(1)}}$ and $f^\delta\bydef f+R_{\overline{f(1)}}$.

The Kantor Lie algebra $\cL=\kan(\cX,\invol)$ is $5$-graded, i.e., has a grading by $\ZZ$ with support $\{-2,-1,0,1,2\}$:
$
\cL=\cL_{-2}\oplus\cL_{-1}\oplus\cL_0\oplus\cL_1\oplus\cL_2,
$
where $\cL_{-2}=(0\times\sks)\tilde{\ }$, $\cL_{-1}=(\cX\times 0)\tilde{\ }$, $\cL_0=\str(\cX,\invol)$, $\cL_1=\cX\times 0$ and $\cL_2=0\times\sks$.
Any grading on $(\cX,\invol)$ by a group $G$ induces naturally a grading by $\ZZ\times G$ on $\kan(\cX,\invol)$.
For the Brown algebra, $\kan(\cA,\invol)$ is the simple Lie algebra of type $E_8$ (see \cite{A79} and note that, as for $\str(\cA,\invol)$, the arguments are valid in characteristic $\ne 2,3$), and we obtain a grading by $\ZZ\times\ZZ_4^3$ on $E_8$, which is the grading that prompted this study of the $\ZZ_4^3$-gradings on the Brown algebra.

Finally, the \emph{Steinberg unitary Lie algebra} $\stu_3(\cX,\invol)$ (see \cite{AF93}) is defined as the Lie algebra generated by the symbols $u_{ij}(x)$, $1\leq i\ne j\leq 3$, $x\in \cX$, subject to the relations:
\[
\begin{split}
&u_{ij}(x)=u_{ji}(-\bar x),\\
&x\mapsto u_{ij}(x)\ \text{is linear,}\\
&[u_{ij}(x),u_{jk}(y)]=u_{ik}(xy)\ \text{for distinct $i,j,k$.}
\end{split}
\]
Then it is easy to see (\cite[Lemma 1.1]{AF93}) that there is a decomposition
\[
\stu_3(\cX,\invol)=\frs\oplus u_{12}(\cX)\oplus u_{23}(\cX)\oplus u_{31}(\cX),
\]
with $\frs=\sum_{i<j}[u_{ij}(\cX),u_{ij}(\cX)]$, which is a grading of $\stu_3(\cX,\invol)$ by $\ZZ_2^2$. Moreover, any grading by a group $G$ on $(\cX,\invol)$ induces naturally a grading by $\ZZ_2^2\times G$ on $\stu_3(\cX,\invol)$.

An explicit isomorphism can be constructed between the quotient of $\stu_3(\cX,\invol)$ by its center and $\kan(\cX,\invol)$ (see \cite{AF93,EldOku_S4}). If $\chr\FF\ne 2,3,5$, then the Killing form of $E_8$ is nondegenerate, so it has no nontrivial central extensions, hence $\stu_3(\cA,\invol)$ is isomorphic to $\kan(\cA,\invol)$, which is the simple Lie algebra of type $E_8$. Thus we obtain a grading by $\ZZ_2^2\times\ZZ_4^3$ on $E_8$. Actually, a Lie algebra $\cK(\cX,\invol,\cV)=\cV\oplus u_{12}(\cX)\oplus u_{23}(\cX)\oplus u_{31}(\cX)$ is defined in \cite[Section 4]{AF93} assuming $\chr\FF\ne 2,3$. Any grading by $G$ on $(\cX,\invol)$ induces a grading by $\ZZ_2^2\times G$ on $\cK(\cX,\invol,\cV)$. For suitable $\cV$, this Lie algebra is isomorphic to $\kan(\cX,\invol)$ (see \cite{AF93,EldOku_S4}), so we obtain a grading by $\ZZ_2^2\times\ZZ_4^3$ on $E_8$ in any characteristic different from $2,3$.

\section*{Acknowledgments}

The first author would like to thank the Department of Mathematics and Statistics of the Memorial University of Newfoundland for hospitality during his visit in August--November 2013. Reciprocally, the third author would like to thank the Instituto Universitario de Matem\'aticas y Aplicaciones and Departamento de Matem\'aticas of the University of Zaragoza for support and hospitality during his visit in January--April 2013.



\begin{thebibliography}{Gar01}

\bibitem[AM99]{AM99} H. Albuquerque and S. Majid, \textit{Quasialgebra structure of the octonions}, J. Algebra \textbf{220} (1999), no.~1, 188--224.

\bibitem[All78]{A78}
B.N. Allison, \textit{A class of nonassociative algebras with involution containing the class of Jordan algebras}, Math. Ann. \textbf{237} (1978), 133--156.

\bibitem[All79]{A79}
B.N. Allison, \textit{Models of isotropic simple Lie algebras}, Comm. Algebra \textbf{7} (1979), no.~17, 1835--1875.

\bibitem[All90]{Allison90}
B.N.~Allison, \textit{Simple structurable algebras of skew-dimension one}.
Comm. Algebra \textbf{18} (1990), no.~4, 1245--1279.

\bibitem[AF84]{AF84}
B.N. Allison and J.R. Faulkner, \textit{A Cayley--Dickson process for a class of structurable algebras}, Trans. Amer. Math. Soc. \textbf{283} (1984), no. 1,  185--210.

\bibitem[AF92]{AF92}
B.N. Allison and J.R. Faulkner, \textit{Norms on structurable algebras}, Comm. Algebra \textbf{20} (1992), no.~1, 155--188.

\bibitem[AF93]{AF93}
B.N. Allison and J.R. Faulkner, \textit{Nonassociative coefficient algebras for Steinberg unitary Lie algebras},
J. Algebra \textbf{161} (1993), no.~1, 1--19.

\bibitem[Bro63]{B63}
R. B. Brown, \textit{A new type of nonassociative algebra}, Proc. Nat. Acad. Sci. U.S.A. \textbf{50} (1963), 947--949.

\bibitem[DV12]{DV_e6}
C. Draper and A. Viruel, \textit{Fine gradings on $\mathfrak{e}_6$}, arXiv:1207.6690

\bibitem[Eld13]{E13} A. Elduque, \textit{Fine gradings and gradings by root systems on simple {L}ie algebras},  arXiv:1303.0651

\bibitem[EK13]{EKmon} A. Elduque and M. Kochetov, \textit{Gradings on simple {L}ie algebras}, Mathematical Surveys and Monographs \textbf{189},  American Mathematical Society, Providence, RI, 2013.

\bibitem[EO07]{EldOku_S4}
A. Elduque and S. Okubo, \textit{Lie algebras with $S_4$-action and structurable algebras},
J. Algebra \textbf{307} (2007), no.~2, 864--890.

\bibitem[Gar01]{G01}
S. Garibaldi,  \textit{Structurable Algebras and Groups of Type $E_6$ and $E_7$}, J. Algebra \textbf{236} (2001), no.~2, 651--691.

\bibitem[Jac68]{J68}
N. Jacobson, \textit{Structure and representations of Jordan algebras}, American Mathematical Society Colloquium Publications \textbf{39}, American Mathematical Society, Providence, RI, 1968.

\bibitem[McC69]{MC69} K. McCrimmon, \textit{The Freudenthal-Springer-Tits constructions of exceptional Jordan algebras}, Trans. Amer. Math. Soc. \textbf{139} (1969), 495--510.

\bibitem[McC70]{MC70} K. McCrimmon, \textit{The Freudenthal-Springer-Tits constructions revisited}, Trans. Amer. Math. Soc. \textbf{148} (1970), 293--314.

\bibitem[Smi92]{Smi92} O.N. Smirnov, \textit{Simple and semisimple structurable algebras}, Proceedings of the International Conference on Algebra, Part 2 (Novosibirsk, 1989), Contemp. Math. \textbf{131}, Part~2, Amer. Math. Soc., Providence, RI, 1992, pp. 685--694.


\end{thebibliography}
\end{document}